\documentclass[10pt]{article}

\usepackage{url}
\usepackage{mathtools}
\usepackage{amssymb}
\usepackage{amsthm}
\usepackage{amsmath}
\usepackage{empheq}
\usepackage{latexsym}
\usepackage{enumitem}
\usepackage{eurosym}
\usepackage{dsfont}
\usepackage{appendix}
\usepackage{color} 
\usepackage[unicode]{hyperref}
\usepackage{frcursive}
\usepackage[utf8]{inputenc}
\usepackage[T1]{fontenc}
\usepackage{geometry}
\usepackage{multirow}
\usepackage{todonotes}
\usepackage{lmodern}
\usepackage{anyfontsize}
\usepackage{stmaryrd}
\usepackage{cleveref}
\usepackage{natbib}
\usepackage[english]{babel}
\usepackage[english=british]{csquotes}
\usepackage{tikz}
\usepackage{pgfplots}
\pgfplotsset{compat=1.16} 
\usepackage{float}

\setcitestyle{numbers,open={[},close={]}}

\definecolor{red}{rgb}{0.7,0.15,0.15}
\definecolor{green}{rgb}{0,0.5,0}
\definecolor{blue}{rgb}{0,0,0.7}
\hypersetup{colorlinks, linkcolor={red},citecolor={green}, urlcolor={blue}}
			
\makeatletter \@addtoreset{equation}{section}

\newtheorem{theorem}{Theorem}[section]
\newtheorem{assumption}[theorem]{Assumption}

\newtheorem{lemma}[theorem]{Lemma}
\newtheorem{proposition}[theorem]{Proposition}

\newtheorem{definition}[theorem]{Definition}
\newtheorem{remark}[theorem]{Remark}

\DeclareUnicodeCharacter{014D}{\=o}
\def \E{\mathbb{E}}
\def \F{\mathbb{F}}

\def \N{\mathbb{N}}

\def \P{\mathbb{P}}
\def \Q{\mathbb{Q}}
\def \R{\mathbb{R}}

\def\Ac{{\cal A}}

\def\Cc{{\cal C}}

\def\Ec{{\cal E}}
\def\Fc{{\cal F}}

\def\Pc{{\cal P}}

\def\Rc{{\cal R}}

\def\Tc{{\cal T}}
\def\Uc{{\cal U}}

\def\Zc{{\cal Z}}

\DeclareMathOperator*{\argmin}{argmin}

\DeclareMathOperator*{\esssup}{ess\,sup}

\title{Pollution regulation for electricity generators in a transmission network\footnote{The authors gratefully acknowledge the support of the ANR project PACMAN ANR-16-CE05-0027, the PGIF project `Massive entry of renewable energy in Chile: operation, storage and intermittency', the ANID FONDECYT/POSTDOCTORADO/3201005, and Basal Program CMM-AFB 170001 from ANID (Chile). The authors would also like to thank Ren\'e A\"{i}d and Thibaut Mastrolia for helpful discussions on the model of the paper. The authors would like to thank Marcelo Matus for helping with the calibration of the numerical simulations.}}

\author{    Nicol\'as {\sc Hern\'andez Santib\'a\~nez} \thanks{Center for Mathematical Modeling, Universidad de Chile, Santiago, Chile. nhernandez@dim.uchile.cl.}
 \and Alejandro {\sc Jofr\'e}\footnote{Center for Mathematical Modelling and Dept. of Mathematical Engineering University of Chile, Casilla 170/3, Correo 3, Santiago, Chile, ajofre@dim.uchile.cl}  
 \and Dylan {\sc Possama\"{i}} \footnote{ETH Z\"urich, Department of Mathematics, R\"amistrasse 1001, 8092 Z\"urich, Switzerland, possamai@math.ethz.ch.}}

        \date{\today}

\begin{document}

\maketitle

\begin{abstract}
In this paper we study a pollution regulation problem in an electricity market with a network structure. The market is ruled by an independent system operator (ISO for short) who has the goal of reducing the pollutant emissions of the providers in the network, by encouraging the use of cleaner technologies. The problem of the ISO formulates as a contracting problem with each one of the providers, who interact among themselves by playing a stochastic differential game. The actions of the providers are not observable by the ISO which faces moral hazard. By using the dynamic programming approach, we represent the value function of the ISO as the unique viscosity solution to the corresponding Hamilton--Jacobi--Bellman equation. We prove that this solution is smooth and characterise the optimal controls for the ISO. Numerical solutions to the problem are presented and discussed. We consider also a simpler problem for the ISO, with constant production levels, that can be solved explicitly in a particular setting. 

\vspace{5mm}

\vspace{5mm}


\end{abstract}

\section{Introduction}

One of the major issues related to efforts to mitigate global warming is related to how one can integrate pollution limits in the dynamic of energy markets. Climate change is a threat for future generations and the world is taking actions, such as the Paris agreement or the more recent COP 26, to reduce the associated risks and consequences. With the goal of limiting the increase in the global average temperature to well below 2$^{\circ}$C above pre-industrial levels and to pursue efforts to limit the temperature increase to 1.5$^{\circ}$C, all countries must find ways to reduce emissions as soon as possible.\footnote{For more details, see \url{http://unfccc.int/resource/docs/2015/cop21/eng/l09r01.pdf}.}  In the energy industry, one way this can be achieved is by creating incentives for producers to use and develop cleaner technologies.

\medskip
A common design in liberalised electricity markets involves wholesale trading through a completely integrated structure, in which generators and consumers participate in auctions.  In the bid-based market, the power generators submit cost functions and a central agent, referred to as independent system operator (ISO), determines the rules for market clearing, while optimising the system operations. In the standard model (see for instance \citeauthor*{escobar2010monopolistic} \cite{escobar2010monopolistic}), the ISO runs a minimum cost program to distribute the total power production among the available producers. It is assumed that the generators are distributed in a transmission network, connected to the locations where demand is concentrated. This is an important feature of the problem, since the structure of the network and its physical properties restrict the choices available to the ISO. Once the generators bid their cost functions, the ISO determines each of their productions and transmissions in the network.

\medskip
In this paper we propose a model in which the ISO also provides incentives to the firms in the market to reduce pollution. We assume that the rules of the auction are of public knowledge, and we focus on the second part of the process, when the generators have already bid their cost functions and the ISO must allocate production. The optimisation performed by the ISO is subject to demand and capacity constraints, and the shape of the network determines the feasible transmissions. As in \cite{escobar2010monopolistic}, we include losses when power is transmitted trough the network. Compared to the literature, we generalise the problem of the ISO in two ways: firstly, by incorporating to its objective function the social cost of pollution, and secondly by making the ISO offer a contract to each energy provider with incentives to reduce the pollutant emissions associated to their production. More precisely, each provider will receive a remuneration (or fine) depending on the pollution level in the environment. The  energy providers will therefore perform private efforts to reduce their emissions, such as acquiring modern devices to control pollution or investing in the use of cleaner technologies, and moral hazard will arise in the relationship between the producers and the ISO. We model the efforts of the providers as the percentage by which they reduce their emissions, which is bounded from below by a constant depending on the characteristics on the firm, mainly the technologies it uses to produce electricity.

\medskip
The methodology we use is the dynamic programming approach in contract theory, applied to a layer of agents interacting through a Nash equilibrium. Given any contract offered by the ISO, we identify the set of Nash equilibria of the producers as solutions to a multidimensional backward stochastic differential equation (BSDE for short). This allows to reformulate the problem of the ISO as a standard stochastic control problem in which it provides incentives to the producers by controlling their certainty equivalent processes, which become state processes in this formulation. An important point of our work is that, given the high dimensional problem of the principal\footnote{It has $2N+1$ states variables, where $N$ is the number of generators in the network.}, which is difficult to treat theoretically and numerically, we manage to prove the smoothness of the value function and that it corresponds to the unique viscosity solution of the Hamilton--Jacobi--Bellman equation associated to the control problem. As a consequence we obtain the existence of an optimal contract, a type of result which is not common in the literature. We also provide a benchmark setting in which an intermediate problem for the ISO can be solved explicitly. 

\medskip
The full problem of the ISO is approached numerically, set on a simplified version of the Chilean market. We show that the increase of pollution can be reduced considerably (around thirty percent in a three-months period) if the ISO signs a contract with each provider in the network. Such contract will cover the cost of production and the effort to reduce emissions and will penalise the pollution levels. The form of the optimal contract makes it easy to be implemented, since the ISO just need to observe dynamically pollution levels and adjust accordingly the payments/fines through the control of a sensibility process.\footnote{Namely, the process $Z$ in Equation \eqref{eq:optimal-contract}.} As a consequence, the social cost is reduced to less than half of its value in the absence of regulation. We compute also the production costs when there is no regulation to have a measure of the compensation that a private ISO, who does not value pollution, would have to receive (for instance from the government) to execute our program. The pollution cost turns out to be much higher than production costs which means that a private entity would have to be almost completely compensated.  Finally, we show that as moral hazard is stronger, the effect of the contracts diminishes. In the limit case the producers will perform practically no reduction efforts and the productions and transmissions in the network will be constant in time. 

\medskip
\textbf{Related literature}. There are some works that study the effect of network constraints on the electricity market. \citeauthor*{borenstein2000competitive} \cite{borenstein2000competitive} show that the capacity of transmission lines can determine the degree of competition of the generators. \citeauthor*{escobar2010monopolistic} \cite{escobar2010monopolistic} study the problem of the ISO in a network with resistance losses, when the goal of the ISO is to minimise the total cost of production. The authors prove that resistance losses matter, as they affect the competition between the producers and allow them to bid higher marginal costs in the auction. 

\medskip
The first, seminal paper on principal--agent problems in continuous-time is by \citeauthor*{holmstrom1987aggregation} \cite{holmstrom1987aggregation}. \citeauthor*{sannikov2007agency} \cite{sannikov2007agency} was the first to use, in continuous-time, the continuation value of the agent as a state variable for the problem of the principal. This idea was formalised by \citeauthor*{cvitanic2018dynamic} \cite{cvitanic2018dynamic}, where the authors develop the dynamic programming approach for principal--agent problems. There is an extensive literature on pollution regulation under moral hazard, which we briefly revisit. \citeauthor*{segerson1988uncertainty} \cite{segerson1988uncertainty} studies incentives schemes for static dispersed pollution problems, with per-unit and lump sum taxes, in a context with unobservable actions. In a non-static setting, \citeauthor*{xepapadeas1992environmental} \cite{xepapadeas1992environmental} designs inter-temporal incentives schemes and discusses the inefficiency of adopting static rules to a dynamic framework of pollution accumulation. \citeauthor*{athanassoglou2010dynamic} \cite{athanassoglou2010dynamic} extends the previous models to the stochastic setting and studies a differential game of pollution control with polynomial profit functions. \citeauthor*{chambers1996non} \cite{chambers1996non} consider a static problem of regulation of diffuse emissions which formulates as a multi-task principal--agent problem. \citeauthor*{bontems2006regulating} \cite{bontems2006regulating} study the static third-best problem of nitrogen pollution regulation under hidden information and moral hazard. In the recent years, pollution regulation in the principal--agent model has been studied by \citeauthor*{la2019dynamic} \cite{la2019dynamic}, in the context of agriculture and extractive industries, where the regulator provides incentives to the firms to reduce their diffuse emissions. \citeauthor*{aid2021optimal} \cite{aid2021optimal} study the problem of a regulator allocating emission allowances to the firms, with the goal of reducing the total carbon emissions. The main differences between these two works and the present paper, is the network structure that we assume for the providers, and the absence of moral hazard in \cite{aid2021optimal}. Moreover, the authors assume in \cite{la2019dynamic} that it is socially optimal that every firm performs maximal effort, which we do not.

\medskip
Our problem falls in the category of a contracting problem with multiple agents. \citeauthor*{holmstrom1982moral} \cite{holmstrom1982moral} was the first to study moral hazard with many agents, in discrete-time. \citeauthor*{elie2019contracting} \cite{elie2019contracting} studied the continuous--time problem and linked the Nash equilibria of the agents to the solutions to a multidimensional BSDE. \citeauthor*{hubert2020continuous} \cite{hubert2020continuous} extends the dynamic programming approach to a hierarchical setting with multiple agents and managers. 

\medskip
Other works where incentives in the energy market are studied include \citeauthor*{alasseur2020adverse} \cite{alasseur2020adverse}, where a dynamic pricing of electricity is designed for a population of heterogeneous clients. \citeauthor*{aid2018optimal} \cite{aid2018optimal} design electricity demand response contracts which impact both the average consumption and its variance. \citeauthor*{elie2021mean} \cite{elie2021mean} study also the problem of demand response contracts, by considering a continuum of consumers with mean-field interaction and consumption with common noise. \citeauthor*{jaimungal2021mean} \cite{jaimungal2021mean} study the regulatory problem in a market with mean-field agents interacting through solar renewable energy certificates.  \citeauthor*{campbell2021deep}  \cite{campbell2021deep} develop a deep learning algorithm to solve principal--agent mean-field games with market-clearing conditions, such as the one in \cite{jaimungal2021mean}.

\medskip
The paper is organised as follows. \Cref{sec-model} describes the model, the optimisation problem faced by the ISO and the game played by the producers. The characterisation of the Nash equilibria of the game and the solution to the problem of the ISO are presented in \Cref{sec-solving-nash} and \Cref{sec-solving-iso} respectively. In \Cref{sec-example} we present a simpler problem for the ISO in a setting that allows to find explicit solutions. \Cref{sec-numerics} provides numerical solutions for the general problem.



\medskip
\textbf{Notations:} We let $\N$ be the set of integers, $\N^\star$ the set of positive integers, and $\R_+$ the set of non-negative real numbers. For any $n\in\N^\star$, $i\in\{1,\dots,n\}$ and any vector $v\in\R^n$, we denote by $v^i$ the $i$-th coordinate of $v$ and by $v^{-i}$ the vector obtained by suppressing the $i$-th coordinate of $v$. For $u\in\R$, we denote by $u\otimes_i v$ the vector in $w\in\R^{n+1}$ whose $i$-th coordinate is equal to $u$ and such that $w^{-i}=v$. We use the same notation for stochastic processes. $^\top$ denotes the transpose operation in $\R^N$. For a function $v:\R\times\R^n \rightarrow \R$ with arguments $(t,\ell)$ we denote by $v_t$ and $v_\ell$ its partial derivatives. For a compact set $C\subset\R^n$, we denote by $\Pi_C:\R^n\longrightarrow C$ the projection function over $C$.

\section{The model}\label{sec-model}

We model the second part of the power auction, when the providers have already bid their cost functions. We aim at solving the ISO's problem, in which it decides the amount of power that each producer has to generate, and the non-negative flows they will send to each other, but this time giving incentives to the producers to reduce their pollutant emissions. We model the problem in continuous-time, over a finite horizon $[0,T]$, where $T>0$ is the maturity of the regulation.

\medskip
{\bf $\bullet$ The network.} We consider a network structure, where each node represents one of the producers who can generate power and send flow to its neighbours over transmission lines. Let thus $(V,E)$ denote an oriented graph, where $V$ is the set of vertexes and $E\subset V\times V$ is the set of edges. We consider a finite network so we write without loss of generality $V=\{1,\dots,N\}$, where $N\in\N^\star$ is the number of producers (nodes). The edges of the graph represent the transmission lines and are denoted by $e\in E$.

\medskip
We assume that each node $i\in\{1,\dots,N\}$ has a power demand $D_i \geq 0$, and power can be sent through the available lines in the network---that is, the edges in the graph. Each producer $i\in\{1,\dots,N\}$ is responsible for generating the amount of power $q_t^i$ and send the flow $\phi^e_t$ through the edge $e\in E$ at time $t\in[0,T]$. Both $q_t^i$ and $\phi^e_t$ are decided by the ISO. The dispatching problem of the ISO is subject to nodal balances, generation and transmission constraints, which we now describe.

\medskip
{\bf $\bullet$ Constraints in the network.} First, the production plan must satisfy the demand at each node. There are power flow losses in the transmission lines, that we approximate by a quadratic function. If the flow over $e\in E$ is $\phi^e$, the loss is given by $r_e (\phi^e)^2$, where $r_e\geq 0$ is the line resistance. We assume the loss is split equally between the two nodes, each one of them suffering half of it. Let $K_i$ be the set of edges connecting node $i$, and $\mathrm{sgn}(e,i)$ is equal to 1 or -1 depending on whether $e$ enters node $i$ or not. Then the plan must satisfy
\begin{equation}\label{eq:balance-constraint}
\frac12\sum_{e\in K_i} r_e (\phi^e_t)^2 + D_i = q_t^i + \sum_{e\in K_i} \phi^e_t \mathrm{sgn}(e,i), \;\forall i\in\{1,\dots,N\}, \; \forall t\in[0,T].
\end{equation}

Second, each producer has a capacity constraint denoted by $Q^i\geq 0$, so that
\begin{equation} \label{eq:production-constraint}
q_t^i \in [0,Q^i], \; \forall i\in\{1,\dots,N\}, \;\forall t\in[0,T].
\end{equation}

Third, each transmission line $e\in E$ has a maximum safe capacity characterised by the quantities $0\leq \underline{\phi}^e\leq \overline{\phi}^e$ so that
\begin{equation} \label{eq:flow-constraint}
\phi^e_t \in [\underline{\phi}^e,\overline{\phi}^e], \; \forall i\in\{1,\dots,N\}, \;\forall t\in[0,T].
\end{equation}

As an example, in Figure \ref{fig:example} the network is given by $V=\{1,2\}$ and $E=\{(1,2)\}$. The production $q^1 \equiv 6$, $q^2 \equiv 9$ together with the flow $\phi^{1,2}\equiv 1$ satisfy the constraints. We see that node 2 does not have the capacity to satisfy its own demand and it can cover it by receiving a flow from node 1.

\begin{figure}[H] 
\begin{center}
		\includegraphics[scale=.7]{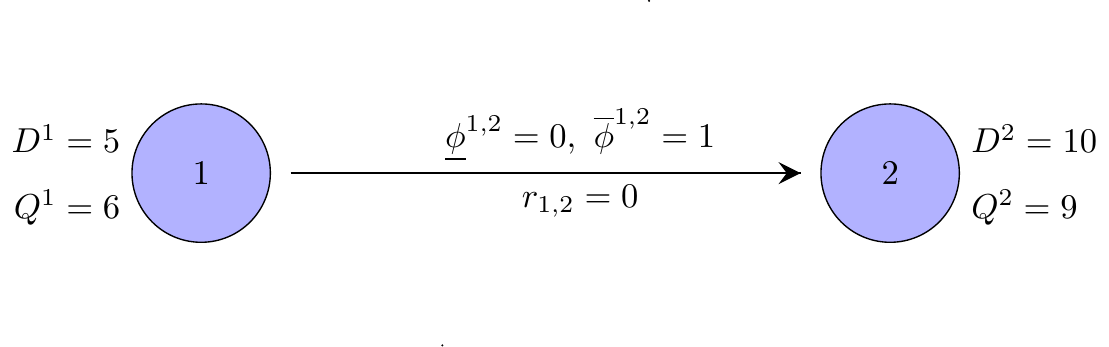} 	
\caption{Example of the network constraints.} \label{fig:example}
\end{center}
\end{figure}
The cost of production at node $i$ is given by the function $c_i:\R_+\longrightarrow\R_+$. In a standard power auction, the ISO determines the production plan $(q^i)_{i\in\{1\,\dots,N\}}$ and the flows $(\phi^e)_{e\in E}$ that minimise the overall cost of production and satisfy the network constraints \eqref{eq:balance-constraint}, \eqref{eq:production-constraint}, and \eqref{eq:flow-constraint}. In our model, the ISO takes pollution levels into consideration and also encourages producers to reduce their emissions, as we detail now.

\medskip
{\bf $\bullet$ Pollution and moral hazard.} {\it A priori}, for any $i\in\{1,\dots,N\}$, if producer $i$ generates an amount $q_t^i$ of power at time $t\in[0,T]$, it will contribute to the pollution in the environment by the amount $p_i(q_t^i)$, where $p_i:\R_+\longrightarrow\R_+$ is the polluting function of producer $i$. Therefore, denoting by $L$ the total pollution process, we assume that it will follow the dynamics
\[
\mathrm{d}L_t = \sum_{i=1}^N p_i(q_t^i) \mathrm{d}t + \sigma \mathrm{d}W^0_t,
\]
where $W^0$ is a standard Brownian motion, intended to represent randomness in the weather condition that cannot be controlled, and $\sigma>0$. Even if in theory the process $L$ can take negative values, the probability of such event occurring decreases with the values of the initial pollution $L_0$ and the polluting functions $p_i$. Given the real-life values of the modelling parameters that we have in mind, we justify the choice of this simple dynamic for the pollution with the fact that, as shown in Section \ref{sec-numerics}, the process $L$ will never become negative in our numerical simulations. 

\medskip
The ISO offers contracts to each of the producers, with incentives to reduce their pollutant emissions. We suppose that producer $i\in\{1,\dots,N\}$ can exert an effort process $a^i$ to reduce its own pollution level, with values in a set $A_i\subset[0,1]$, and associated to a second cost function $h_i:A_i \longrightarrow \R_+$. When all the producers choose their efforts, they impact the distribution of the pollution process $L$ as follows
\[
\mathrm{d}L_t = \sum_{i=1}^N  (1-a^i_t)p_i(q_t^i) \mathrm{d}t + \sigma \mathrm{d}W^a_t,
\]
where the process $W^a$ is a Brownian motion under the measure $\P^a$ induced by the joint efforts of the producers. This is a weak formulation of the problem, which is standard in contract theory. Notice that the process $L$ is not controlled directly, but its distribution. All the details on the construction of the weak formulation are given in \Cref{ap:model}, as well as the set of admissible actions $\Ac_i$ and joint actions $\Ac$ for the $i$-th producer and all them respectively.

\medskip
{\bf $\bullet$ The regulation contract.} The ISO chooses a production plan $q=(q^1,\dots,q^N)^\top$ and transmission plan $\phi=(\phi^e)_{e\in E}$ for the producers and offers terminal remunerations $\xi=(\xi^1,\dots,\xi^N)^\top$ according to the pollution in the environment. That is, the production and transmission processes are adapted to $\mathbb{F}:=(\mathcal{F}_t)_{t\in [0,T]}$, the (completed) filtration generated by $L$, and the remunerations are $\mathcal{F}_T$-measurable random variables.\footnote{See Appendix \ref{ap:model} for the definition of the filtration $\mathbb{F}$.} In particular, efforts of each producer to reduce their emissions are unobservable by the ISO, which thus faces \emph{moral hazard}. We refer to a triplet $(q,\phi,\xi)$ as a contract.

\medskip
{\bf $\bullet$ Game between the producers.} We suppose the ISO cannot distinguish the individual contributions of each producer to the total pollution. Since the contracts are written in terms of the process $L$, this means the actions of a particular producer affect the welfare of all of them. The producers play an $N$-player differential game when deciding their efforts and we assume they look for a Nash equilibrium of the game.

\medskip
Given a production and transmission plan $(q,\phi)\in\Pc$ and remunerations $\xi\in\Rc$\footnote{The sets $\Pc$ and $\Rc$ just mentioned are defined in the next section.}, if the $N$ producers perform the joint action $a\in\Ac$, the utility obtained by agent $i\in\{1,\dots,N\}$ is equal to 
\[
U^i_0(a^i,a^{-i},\xi^i,q):= \E^{\P^{a}} \bigg[ \Uc_A\bigg( \xi^i - \int_0^T  \big(h_i(a_s^i)+c_i(q_s^i) \big) \mathrm{d}s \bigg) \bigg].
\]
The best response of producer $i$ to the actions $a^{-i}\in\Ac^{-i}$ of the others is obtained by solving the following problem
\begin{equation} \label{eq:producer-problem}
V_0^i(a^{-i},\xi^i,q) := \sup_{a^i\in\Ac_i}  \E^{\P^a} \bigg[ \Uc_A \bigg( \xi^i - \int_0^T  \big(h_i(a_s^i)+c_i(q_s^i) \big)\mathrm{d}s  \bigg) \bigg],
\end{equation}
with the CARA utility function $\Uc_A:\R\longrightarrow \R$ given by $\Uc_A(x):=-\mathrm{e}^{-\rho x}$, $x\in\R$, for some $\rho>0$. 

\medskip
{\bf $\bullet$ Social cost and constraints of the ISO.} On the other hand, the goal of the ISO is to minimise the social cost, taking into account the level of pollution. Therefore, the ISO solves the following problem
\begin{equation}\label{eq:ISO-problem}
\inf_{(q,\phi,\xi)\in\Cc}  \E^{\P^{a^\star(q,\xi)}} \bigg[ \sum_{i=1}^N \int_0^T  c_i(q^i_s) \mathrm{d}s + \int_0^T \Lambda(L_s - \ell_o) \mathrm{d}s +  \sum_{i=1}^N  \xi^i \bigg],
\end{equation}
where $a^\star(q,\xi)$ is a Nash equilibrium\footnote{We prove in Section \ref{sec-solving-nash} the existence of a unique Nash equilibrium to the contract, and it does not depend on the flow $\phi$.} to the contract $(q,\phi,\xi)$ that the producers agree on playing (more details are given in Section \ref{sec-solving-nash}), $\ell_o>0$ is the target level of pollution and $\Lambda:\R\longrightarrow\R$ is the cost of deviating from said target. We list the properties assumed for the function $\Lambda$, as well as the functions $(c_i, p_i, h_i)_{i\in\{1,\dots,N\}}$, in \Cref{ap:model}.

\medskip
The optimisation problem of the ISO is subject to constraints \eqref{eq:balance-constraint}, \eqref{eq:production-constraint}, \eqref{eq:flow-constraint}. Additionally, in order to agree to this contract, each producer must obtain a minimum value of expected utility denoted by $R_0^i$
\begin{align}\label{eq:participation-constraint}
V_0^i(a^{\star,-i},\xi^i,q) \geq R_0^i, \;\forall i\in\{1,\dots,N\}.
\end{align}
Notice that we have not yet specified the space of controls $\Cc$ over which the ISO optimise. We do this in the next section so Problem \eqref{eq:ISO-problem} is properly defined.

\begin{remark}
There are different possibilities for the values $(R_0^i)_{i\in\{1,\dots,N\}}$. We take them as exogenous, representing a reservation utility that the producers can obtain if they do not enter the auction. One could also take, for any $i\in\{1,\dots,N\}$, $R_0^i$ as the endogenous value the producer $i$ would obtain without the pollution regulation of the {\rm ISO}, in a standard auction. Such problem is the one studied in {\rm\cite{escobar2010monopolistic}}.
\end{remark}

\begin{remark}\label{rm-piecewise}
The setting that we have in mind is the one in which each provider has different technologies for producing power. In this case, the cost functions $(c_i)_{i\in\{1,\dots,N\}}$ will usually be piecewise linear, and each provider will produce at the lowest marginal cost until the cheapest available technology is saturated. Assuming that the pollution of each technology is linear, it follows that each function $(p_i)_{i\in\{1,\dots,N\}}$, impacting the drift of the process $L$, is also piecewise linear.
\end{remark}

\subsection{Nash equilibria and problem of the ISO} \label{sec-solving-nash}

In this section we define the set of Nash equilibria associated to any given contract. This is important for writing formally the optimisation problem of the ISO. We proceed then to provide a characterisation of Nash equilibria through multidimensional BSDEs which allows us to reformulate the problem of the ISO as a standard stochastic control problem.

\medskip
Before doing so, let us define the set of admissible remunerations $\mathcal{R}$ as the ones satisfying the following integrability condition\footnote{See Appendix \ref{ap:model} for the definition of the underlying probability space $(\Omega,\Fc,\P)$.} 
\[
\Rc := \Big\{ \xi: \R^N\text{-valued,}\;\Fc_T\text{-measurable}\; \text{r.v.}\; \text{such}\; \text{that}\;  \E\big[|\Uc_A(\xi^i)|^p\big] < \infty,\; \text{for}\; \text{some}\; p>1,\;  i\in\{1,\dots,N\}     \Big\}.
\]
Next, let $\hat{P}$ be the set of feasible plans, that is 
\[
\hat{P} := \Bigg\{ (q,\phi)\in \R_+^N \times \R_+^E: \forall (e,i)\in E\times\{1,\dots,N\},\; (\phi^e,q^i) \in[\underline{\phi}^e,\overline{\phi}^e]\times [0,Q^i], \; D_i = q^i + \sum_{e\in K_i} \bigg( \phi^e \text{sgn}(e,i) - \frac{r_e}{2} (\phi^e)^2 \bigg) \Bigg\}.
\]
We assume the production problem is feasible, that is, the set $\hat{P}$ is non-empty. We define then $\Pc$ as the set of $\mathbb{F}$-predictable, $\hat{P}$-valued processes.

\medskip
Fix now a contract $(q,\phi,\xi)\in\Pc\times\Rc$ offered by the ISO. If the $N$ producers perform the joint action $a\in\Ac$, recall from the previous section, for any $i\in\{1,\dots,N\}$, the utility obtained by agent $i$, $U^i_0(a^i,a^{-i},\xi^i,q)$, and the best reaction of producer $i$ given the actions of the others, $V_0^i(a^{-i},\xi^i,q)$.

\begin{definition}
Given a contract $(q,\phi,\xi)\in \Pc\times\Rc$, the joint action $a^\star\in\Ac$ is a Nash equilibrium for the producers, denoted by $a^\star\in\mathrm{NE}(q,\phi,\xi)$, if for every $i\in\{1,\dots,N\}$
\[
V_0^i(a^{\star,-i},\xi^i,q)= U_0^i(a^{\star,i},a^{\star,-i},\xi^i,q).
\]
\end{definition}
We will enforce that admissible contracts must generate only one Nash equilibrium. We will see in Theorem \ref{th:main} below that this restriction is without loss of generality. Thus, we define the set of admissible contracts $\Cc$ as
\[
\Cc := \Big\{ (q,\phi,\xi)\in\Pc\times\Rc:  \text{NE}(q,\phi,\xi)=\{ a^\star\},\;  \text{and}\; V_0(a^{\star,-i},\xi^i,q) \geq R_0^i, \;\forall i\in\{1,\dots,N\} \Big \}.
\]
For any contract $(q,\phi,\xi)\in\Cc$, we thus denote by $a^\star(q,\phi,\xi)$ the unique element in $\text{NE}(q,\phi,\xi)$. We can now define formally the optimisation problem of the ISO \eqref{eq:ISO-problem} as
\begin{equation}\label{eq:ISO-problem-2}
V_0:=\inf_{(q,\phi,\xi)\in\Cc} \E^{\P^{a^\star(q,\phi,\xi)}} \bigg[ \sum_{i=1}^N \int_0^T  c_i(q^i_s) \mathrm{d}s + \int_0^T \Lambda(L_s - \ell_o) \mathrm{d}s +  \sum_{i=1}^N  \xi^i \bigg].
\end{equation}

The main difficulties associated with Problem \eqref{eq:ISO-problem-2} are the general form of the remunerations $\xi^i$ and the determination of $a^\star(q,\phi,\xi)$, which result in a non-standard stochastic control problem. By applying the so-called dynamic programming approach, introduced in \cite{cvitanic2018dynamic}, we are able to characterise the Nash equilibrium associated to a contract, and reduce the set of remunerations without loss of generality to a suitable class which allows to associate a Hamilton--Jacobi--Bellman equation to the problem of the ISO.

\medskip
The next theorem summarises our main results, and its proof is deferred to \Cref{ap:main}.

\begin{theorem}\label{th:main}
The set of admissible remunerations can be represented as the terminal values of the following family of processes
\[
Y_t^{y,q,Z} = y + \int_0^t f(q_s,Z_s) \mathrm{d}s + \int_0^t Z_s \mathrm{d}L_s,\; t\in[0,T],
\]
with the function $f:\R_+^N\times \R^N\longrightarrow\R^N$ given by \eqref{eq:generator}. 
That is
\[
\Cc = \Big\{ (q,\phi,Y_T^{y,q,Z}): (q,\phi,Z,y)\in\Pc\times\Zc\times\R^n,\; \text{\rm with}\; \Uc_A(y^i)\geq R_0^i,\; \forall i\in\{1,\dots,N\} \Big\},
\]
where $\Zc$ is the class of processes defined in \eqref{eq:classZ}. For any remuneration with the form $\xi=Y_T^{y,q,Z}$, for some  $(q,\phi,Z,y)\in\Pc\times\Zc\times\R^n$, there exists a unique Nash equilibrium $a^\star$ which satisfies, for every $i\in\{1,\dots,N\}$,  $V_0^i(a^{\star,-i},\xi^i,q)=\Uc_A(y^i)$ and 
\begin{equation} \label{eq:a-star}
\{a^{\star,i}_s\} = \argmin_{a\in A_i} \big\{h_i(a)-Z_s^i (1-a)p_i(q^i_s) \big\},\; \mathrm{d}s\otimes\mathrm{d}\P\text{\rm--a.e.}
\end{equation}
\end{theorem}

\begin{remark}
We see in the previous theorem that the dependence of the Nash equilibrium of the providers on the plan $(q,\phi)\in\Pc$ is only through the production $q$. For this reason, we drop $\phi$ from the notation for the rest of the paper.
\end{remark}

For $(q,\phi)\in\Pc$ and $Z\in\Zc$, we denote by $a^\star(q,Z)$ the Nash equilibrium given by Theorem \ref{th:main} and, abusing notations slightly, by $a^{\star,i}(q_s^i,Z_s^i)$ the minimising function in \eqref{eq:a-star}. We have therefore, the following reformulation of the problem of the ISO as a standard stochastic control problem
\begin{equation}\label{eq:V0-first}
V_0 = \inf_{(q,\phi,y,Z)\in \hat{\mathcal{C}}} \E^{\P^{a^\star(q,Z)}} \bigg[ \sum_{i=1}^N \int_0^T  c_i(q^i_s) \mathrm{d}s + \int_0^T \Lambda(L_s - \ell_0) \mathrm{d}s +  \sum_{i=1}^N  Y_T^{y,q,Z,i} \bigg],
\end{equation}
where the set $\hat\Cc$ is given by the reformulated constraints
\[
\hat \Cc:= \big\{ (q,\phi,y,Z)\in\Pc\times\mathbb{R}^n\times\Zc: \Uc_A(y^i) \geq R_0^i, \;\forall i\in\{1,\dots,N\}\big \}.
\]

\begin{remark}
The controlled process $Y^{y,q,Z}$ allows to tackle the problem of the {\rm ISO} by following the dynamic programming approach.  It represents the certainty equivalent of the producers and it plays the same role as the continuation utility process in principal--agent problems, see for instance {\rm\citeauthor*{sannikov2008continuous} \cite{sannikov2008continuous}}.
\end{remark}

\begin{remark}
The variable $y$ is not really part of the stochastic control problem, since in general, one solves a control problem for any given value of $y$. In our setting the optimisation over $y$ is direct and performed explicitly, see {\rm \Cref{prop:value-function-2}}.
\end{remark}

\subsection{Solving the problem of the ISO} \label{sec-solving-iso}
Theorem \ref{th:main} allows to reformulate the problem of the ISO as a standard stochastic control problem in which the ISO controls directly the production and transmission plans $(q,\phi)$ and controls the remunerations through the processes $(q,Z)$ and the initial values $y$.  The ISO solves \eqref{eq:V0-first}, with the dynamics of the controlled processes given by
\begin{align*}
L_t &= L_0 + \int_0^t \sum_{i=1}^N \big(1-a^{\star,i}(q_s^i,Z_s^i)\big)p_i(q^i_s)  \mathrm{d}s + \sigma \int_0^t \mathrm{d}W^{a^\star(q,Z)}_s, \; t\in[0,T],\\
Y_t^{y,q,Z,i} &  = y^i + \int_0^t \bigg(  h_i\big(a^{\star,i}(q_s^i,Z_s^i)\big) + c_i(q_s^i)+ \frac{\rho\sigma^2}{2} (Z^i_s)^2 \bigg) \mathrm{d}s + \int_0^t Z^i_s \sigma \mathrm{d}W^{a^\star(q,Z)}_s,\; t\in[0,T],\; i\in\{1,\dots,N\}.
\end{align*}
Due to the dynamics of the second controlled process, this problem can be further simplified. It turns out that the appropriate state variable for Problem \eqref{eq:V0-first} is the sum of the certainty equivalents $\hat {Y}^{y,q,Z}:= \sum_{i=1}^N  Y^{y,q,Z,i}$. Moreover, it is immediate to deduce the dependence on $y$ for the value function. We have thus the following equivalence whose proof can be found in Appendix \ref{app:iso}.
\begin{proposition}\label{prop:value-function-2}
The reformulated problem of the {\rm ISO} can be written as 
\[
V_0 = \hat V_0 - \sum_{i=1}^N \frac{1}{r}\log(-R_0^i),
\]
where $\hat V_0$ is the value of the following stochastic control problem
\begin{equation}\label{eq:V0-one-dimensional}
\hat V_0 := \inf_{(q,\phi,Z)\in\Pc\times \Zc}  \E^{\P^{a^\star(q,Z)}} \bigg[ \sum_{i=1}^N \int_0^T  \bigg( h_i\big(a^{\star,i}(q_s^i,Z_s^i)\big)+\frac{\rho\sigma^2}{2} (Z_s^i)^2+ 2 c_i(q^i_s)  \bigg) \mathrm{d}s + \int_0^T \Lambda( L_s - \ell_0) \mathrm{d}s  \bigg] .
\end{equation}

\end{proposition}
The importance of this result is the reduction of the state variables in the reformulated problem. Indeed, \Cref{prop:value-function-2} presents a new problem where the only state variable is the process $L$, which results in a one-dimensional drift-control problem, which is much easier to deal with, compared to the original problem with two state variables, degenerate diffusion coefficient, and controlled volatility.

\medskip
The Hamilton--Jacobi--Bellman partial differential equation associated to the reformulated problem of the ISO \eqref{eq:V0-one-dimensional} is the following
\begin{equation}\label{eq:pde-value-function}
v_t +  G(\ell, v_\ell, v_{\ell\ell}) = 0, \; (t,\ell)\in [0,T) \times \R,\;
 v(T,\ell) =  0,\; \ell\in\R,
\end{equation}
with $G:\R \times \R \times \R \longrightarrow \R$ given by
\begin{equation*} 
 G(\ell,\alpha,\gamma)  := \inf_{(z,q,\phi)\in\R^{N}\times \hat{P} }  \big\{g(\alpha,z,q,\phi)\big\}   + \frac12\gamma\sigma^2  + \Lambda(\ell - \ell_0),\; (\ell,\alpha,\gamma)\in\R^3,
\end{equation*}
and $g:\R\times\R^N\times \hat{P}\longrightarrow \R$ defined by
\[
g(\alpha,z,q,\phi) := \sum_{i=1}^N  \bigg(  \alpha \big(1-a^{\star,i}(q,z)\big)p_i(q^i) +  h_i\big(a^{\star,i}(q,z)\big)+\frac{\rho\sigma^2}{2} (z^i)^2+2c_i(q^i) \bigg),\; (\alpha,z,q,\phi)\in\R\times\R^N\times \hat{P}.
\]
This PDE is well-behaved, in the sense that it possesses a unique viscosity solution, which turns out to be smooth. We have the main result of this section.
\begin{theorem}\label{thr-iso-main} $(i)$ The value function of problem \eqref{eq:V0-one-dimensional} is given by  $\hat V_0=v(0,L_0)$, where $v$ is the unique viscosity solution to the {\rm HJB} equation \eqref{eq:pde-value-function} with polynomial growth at infinity. Moreover, $v$ is continuously differentiable in the space variable.

\medskip
$(ii)$ The optimal controls for problem \eqref{eq:V0-one-dimensional} are given by $Z^\star_s:= z^\star(v_\ell(s,L_s))$, $\phi^\star_s:= \phi^\star(v_\ell(s,L_s))$, $q^\star_s:= q^\star(v_\ell(s,L_s))$, where $(z^\star(\alpha), \phi^\star(\alpha), q^\star(\alpha))$ is any measurable selection of minimisers of $g(\alpha,\cdot)$ $($see {\rm\Cref{lemma-optimizers}}$)$. 

\medskip
$(iii)$ Let $y^i:= -\frac{1}{\rho}\log(-R_0^i)$, $i\in\{1,\dots,N\}$. The optimal contract in {\rm Problem \eqref{eq:ISO-problem}} is given by $(q^\star,\phi^\star,\xi^\star)$, with the remunerations 
\begin{equation}\label{eq:optimal-contract}
\xi^\star:= y + \int_0^T f(q_s^\star,Z_s^\star) \mathrm{d}s + \int_0^T Z_s^\star \mathrm{d}L_s,
\end{equation}
with the function $f:\R_+^N\times\R^N\longrightarrow\R^N$ given by \eqref{eq:generator}.
\end{theorem}


\section{A simpler problem for the ISO}

In this section we discuss a different problem for the ISO, in which the production and transmission plans are fixed throughout the life of the contract. If the ISO is not able or willing to dynamically update the values of the production and transmission, the regulation problem will be slightly different and mathematically simpler. Namely, the decision over the controls $(q,\phi)\in\Pc$ will become a choice over elements $(q,\phi)\in\hat P$. 

\subsection{The new problem}

We keep the same notations from the previous section. We also enforce the same assumptions over the the functions in our model (see \Cref{ap:model}). We study a sub-problem of \eqref{eq:ISO-problem} in which the ISO is restricted to choose constant controls from $\Pc$. What motivates this problem is the fact that in real life, the ISO may not want to constantly update production and transmission due to the operational costs involved. 

\medskip
We introduce the following notation. For $(q,\phi)\in\hat{P}$, we denote by $(q^d,\phi^d)$ the deterministic and constant processes defined by $q^d_t(\omega) \equiv q$, $\phi^d_t(\omega) \equiv \phi$. We study then, the following regulation problem for the ISO
\begin{equation}\label{eq:ISO-sec3}
V^d:= \inf_{(q,\phi)\in\hat P} v^d(q,\phi),
\end{equation}
where the problem with fixed plan $(q,\phi)\in\hat{P}$ is given by 
\begin{equation}\label{eq:ISO-problem-sec3}
v^d(q,\phi) :=  \inf_{\xi\in\Cc(q,\phi)}  \E^{\P^{a^\star(q^{\text{\fontsize{4}{4}\selectfont $d$}},\xi)}} \bigg[ \sum_{i=1}^N   c_i(q^i) T + \int_0^T \Lambda(L_s - \ell_o) \mathrm{d}s +  \sum_{i=1}^N  \xi^i \bigg],
\end{equation}
with the set of remunerations
\[
\Cc(q,\phi) := \Big\{ \xi\in\Rc:  \text{NE}(q^d,\phi^d,\xi)=\{ a^\star\},\;  \text{with}\; V_0(a^{\star,-i},\xi^i,q^d) \geq R_0^i, \;\forall i\in\{1,\dots,N\} \Big \}.
\]
Let us mention that the game played by the producers does not change at all in this new problem, since their actions keep being taken for given production and transmission plans. Therefore, we can prove in an identical way to the proof of \Cref{th:main}, the following equivalence for the set of remunerations
\[
\Cc(q,\phi) = \big\{ Y_T^{y,q,Z}: (Z,y)\in\Zc\times\R^n,\; \text{with}\; \Uc_A(y^i)\geq R_0^i,\; \forall i\in\{1,\dots,N\} \big\},
\]
where $\Zc$ is the class of processes defined in \eqref{eq:classZ}. We can reformulate then, the problem with fixed plan \eqref{eq:ISO-problem-sec3} as
\begin{equation}\label{eq:V0-first-sec3}
v^d(q,\phi)  = \inf_{(y,Z)\in\hat\Cc(q,\phi)} \E^{\P^{a^\star(q^{\text{\fontsize{4}{4}\selectfont $d$}},\xi)}} \bigg[ \sum_{i=1}^N   c_i(q^i) T + \int_0^T \Lambda(L_s - \ell_0) \mathrm{d}s +  \sum_{i=1}^N  Y_T^{y,q,Z,i} \bigg],
\end{equation}
with the dynamics of the controlled processes given by
\begin{align*}
L_t &= L_0 + \int_0^t \sum_{i=1}^N \big(1-a^{\star,i}(q^i,Z_s^i)\big)p_i(q^i)  \mathrm{d}s + \sigma \int_0^t \mathrm{d}W^{a^\star(q^{\text{\fontsize{4}{4}\selectfont $d$}},\xi)}_s, \; t\in[0,T],\\
Y_t^{y,q,Z,i} &  = y^i + \int_0^t \bigg(  h_i\big(a^{\star,i}(q^i,Z_s^i)\big) + c_i(q^i)+ \frac{\rho\sigma^2}{2} (Z^i_s)^2 \bigg) \mathrm{d}s + \int_0^t Z^i_s \sigma \mathrm{d}W^{a^\star(q^{\text{\fontsize{4}{4}\selectfont $d$}},\xi)}_s,\; t\in[0,T],\; i\in\{1,\dots,N\},
\end{align*}
and where the set $\hat\Cc(q,\phi)$ is given by the reformulated constraints
\[
\hat \Cc(q,\phi) = \big\{ (y,Z)\in \mathbb{R}^n\times\Zc: \Uc_A(y^i) \geq R_0^i, \;\forall i\in\{1,\dots,N\}\big \}.
\]
It is straightforward, as in \Cref{prop:value-function-2}, that we have
$
v^d(q,\phi)  = \hat v^d(q,\phi)  - \sum_{i=1}^N \frac{1}{r}\log(-R_0^i),
$
where $\hat v^d(q,\phi)$ is the value of the following stochastic control problem, with only the pollution as state variable
\begin{equation}\label{eq:V0-one-dimensional-sec3}
\hat v^d(q,\phi)  := \inf_{Z \in \Zc}  \E^{\P^{a^\star(q^{\text{\fontsize{4}{4}\selectfont $d$}},\xi)}} \bigg[ \sum_{i=1}^N \int_0^T  \bigg( h_i\big(a^{\star,i}(q^i,Z_s^i)\big)+\frac{\rho\sigma^2}{2} (Z_s^i)^2+ 2 c_i(q^i)  \bigg) \mathrm{d}s + \int_0^T \Lambda( L_s - \ell_0) \mathrm{d}s  \bigg] .
\end{equation}
The HJB partial differential equation associated to the reformulated problem \eqref{eq:V0-one-dimensional-sec3} is the following
\begin{equation} \label{eq:pde-value-function-sec3}
v_t +  G^{q,\phi}(\ell, v_\ell, v_{\ell\ell}) = 0, \; (t,\ell)\in [0,T) \times \R,\;
 v(T,\ell) =  0,\; \ell\in\R,
\end{equation}
with $G^{q,\phi}:\R \times \R \times \R \longrightarrow \R$ given by
\begin{equation*}
 G^{q,\phi}(\ell,\alpha,\gamma)  = \inf_{z \in\R^{N} }  \big\{g^{q,\phi}(\alpha,z)\big\}   + \frac12\gamma\sigma^2  + \Lambda(\ell - \ell_0),
\end{equation*}
and $g^{q,\phi}:\R\times\R^N\longrightarrow \R$ defined by
\[
g^{q,\phi}(\alpha,z) = \sum_{i=1}^N  \bigg(  \alpha  \big(1-a^{\star,i}(q,z)\big)p_i(q^i) +  h_i(a^{\star,i}(q,z))+\frac{\rho\sigma^2}{2} (z^i)^2+2c_i(q^i) \bigg) .
\]
Finally, we can mimic \Cref{thr-iso-main} and obtain the analogous result in this new setting. 
\begin{theorem}\label{thr-iso-main-sec3} $(i)$ The value function of problem \eqref{eq:V0-one-dimensional-sec3} is given by  $\hat v^d(q,\phi)=v^{q,\phi}(0,L_0)$, where $v^{q,\phi}$ is the unique viscosity solution to the {\rm HJB} equation \eqref{eq:pde-value-function-sec3} with polynomial growth at infinity. Moreover, $v^{q,\phi}$ is continuously differentiable in the space variable.

\medskip
$(ii)$ The optimal control for problem \eqref{eq:V0-one-dimensional-sec3} is given by $Z^\star_s:= z^\star(v^{q,\phi}_\ell(s,L_s))$, where $z^\star(\alpha)$ is any measurable selection of minimisers of $g^{q,\phi}(\alpha,\cdot)$. 

\medskip
$(iii)$ Let $y^i:= -\frac{1}{\rho}\log(-R_0^i)$, $i\in\{1,\dots,N\}$. The optimal vector of remunerations in {\rm Problem \eqref{eq:ISO-problem-sec3}} is given by 
\[
\xi^\star:= y + \int_0^T f(q,Z_s^\star) \mathrm{d}s + \int_0^T Z_s^\star \mathrm{d}L_s,
\]
with the function $f:\R_+^N\times\R^N\longrightarrow\R^N$ given by \eqref{eq:generator}.
\end{theorem}

To conclude this section, let us comeback to the main problem $V^d$, which can be approached by standard optimisation techniques. The next result establishes the existence of a solution to $V^d$, its proof can be found in Appendix \ref{ap:example}.

\begin{proposition}\label{prop-continuity-sec3}
The function $v^d:\hat P \longrightarrow \R$ is continuous. There exists a deterministic plan $(q^\star,\phi^\star)\in\hat P$ which minimises  \eqref{eq:ISO-sec3}.
\end{proposition}

Finally, if the map $v^d$ is continuously differentiable, since $\hat P$ satisfies the linear independence qualification constraint, the optimisation in \eqref{eq:ISO-sec3} can be performed by solving the corresponding Karush--Kuhn--Tucker optimality conditions. Such smoothness for $v^d$ can be obtained by strengthening our assumption, for instance if the maps $p_i$, $c_i$, $h_i$ are continuously differentiable, and so is the function $a^\star$ with respect to $q$, then \cite[Proposition 2.4]{el1997backward} can be used to prove that $v^d$ will be continuously differentiable.

\subsection{An example with explicit remunerations} \label{sec-example}

We present in this subsection an example in which the optimal remunerations for the providers in the class $\hat \Cc(q,\phi)$ can be found explicitly. We assume that each provider has only one available technology and therefore the polluting functions can be written as linear functions $p_i(q)=p_i q$, for some constants $p_i>0$, $i\in\{1,\dots,N\}$. 

\medskip
By abusing the notations, we let the set of actions be $A_i=[0,A_i]$ for some $A_i\in(0,1)$. Under these assumptions, for any $(q,\phi)\in\hat P$ the problem $v^d(q,\phi)$ of the ISO can be solved explicitly and we have  the following result, whose proof can be found in Appendix \ref{ap:example}.

\begin{proposition}\label{prop:explicit-example}
For any $i\in\{1,\dots,N\}$, let $\Lambda(x)=\lambda x$, $h_i(a)=\frac{h_i}{2}a^2$, $p_i(q)=p_i q$ with positive constants $\lambda$, $h_i$, and $p_i$. For $(q,\phi)\in\hat P$, suppose that for any $i\in\{1,\dots,N\}$, the constants 
\[
M_i:=\frac{A_i h_i(\sigma^2 \rho h_i + (p_i q_i)^2)}{ (p_i q_i)^3},\; i\in\{1,\dots,N\},
\]
satisfy $M_i \geq \lambda T$. Then the optimal control in problem \eqref{eq:V0-one-dimensional-sec3} is given by
\[
Z_s^\star =  z_i (s-T), \; s\in[0,T],
\]
where  $z_i :=   \lambda (p_i q_i)^2/\big(\sigma^2 \rho h_i + (p_i q_i)^2\big)$. The optimal remunerations in {\rm Problem \eqref{eq:ISO-problem-sec3}} and the corresponding Nash equilibrium of the producers are given by
\begin{align*}
 \xi^{\star,i} := y^i +  \int_0^T f^i(q,Z_s^\star) \mathrm{d}s +  z_i \int_0^T (s-T) \mathrm{d}L_s, \; a^{\star,i}_s= \frac{(T-s) z_i p_i q_i}{h_i }, \; s\in[0,T],\; i\in\{1,\dots,N\}.
\end{align*}
\end{proposition}

Let us comment the form of the optimal remunerations in this case. We see in the last term, that an increase in the pollution level is always penalised, and the penalisation is distributed among the firms according to the constants $(z_i)_{i\in\{1,\dots,N\}}$. The higher the polluting coefficient $p_i$, the higher the constant $z_i$ and the remuneration of the firm is more sensitive to the total pollution. The more production $q_i$ the firm is required to provide, the higher is the constant $z_i$ and the same result follows. On the other hand, the costlier the effort $h_i$ of reducing pollution, the less the sensitivity of the remuneration on the total pollution. We also mention that the ISO covers both the cost of production and the cost of effort of each provider, as we can see in the definition \eqref{eq:generator} of the function $f^i$.

\medskip
Concerning the Nash equilibrium of the producers, we see that efforts are decreasing in time. At the end of the contract, no effort is made by any firm. This is due to our assumption $M_i \geq \lambda T$, meaning that the length of the contract is relatively small. The closer to maturity, the more expensive it is for the ISO to encourage efforts from the producers, and these processes naturally decrease to zero. Mathematically, as we can see in the proof of \Cref{prop:explicit-example}, our assumption makes the space derivative of the value function smaller than $M_i$ and consequently the processes $Z^\star$ and $a^\star$ are decreasing. On the other hand, when the space derivative of the value function is bigger than $M_i$ the processes $Z^\star$ and $a^\star$ are constant, being in the boundary of some sub-domain (see \eqref{eq:G-explicit-example}). In the next section, we present some numerical examples with the latter feature, for which it is optimal for the ISO to encourage constant effort throughout almost the whole life of the contract.

\section{Numerical solutions}\label{sec-numerics}

In this section we present the results obtained by solving numerically the HJB equation \eqref{eq:pde-value-function}. We use the algorithm studied in \citeauthor*{bonnans2004fast} \cite{bonnans2004fast} to approximate the solution to such PDE. 

\medskip
Inspired by the Chilean electricity market\footnote{See \url{https://www.cne.cl/wp-content/uploads/2019/10/RT_Financiero_v201910.pdf} or \url{https://www.cne.cl/wp-content/uploads/2020/01/Ap\%C3\%A9ndice-II-Proyecci\%C3\%B3n-de-Demanda-El\%C3\%A9ctrica-2019-\%E2\%80\%93-2039.pdf}}, we consider a network with three nodes representing the North, South and Centre regions of the country, respectively. We set the demands at each node to $D_1=3000$ MWh, $D_2=1000$ MWh, $D_3=6000$ MWh and the production capacities $Q^1=6000$ MWh, $Q^2=2000$ MWh, $Q^3=12000$ MWh. We suppose all the electricity can be sent through the lines and we choose the resistance values in order to make the loses not go higher than 5\% of the flow. We present the network and its characteristics in Figure \ref{fig:network}.

\begin{figure}[ht!]
\begin{center}
		\includegraphics[scale=.6]{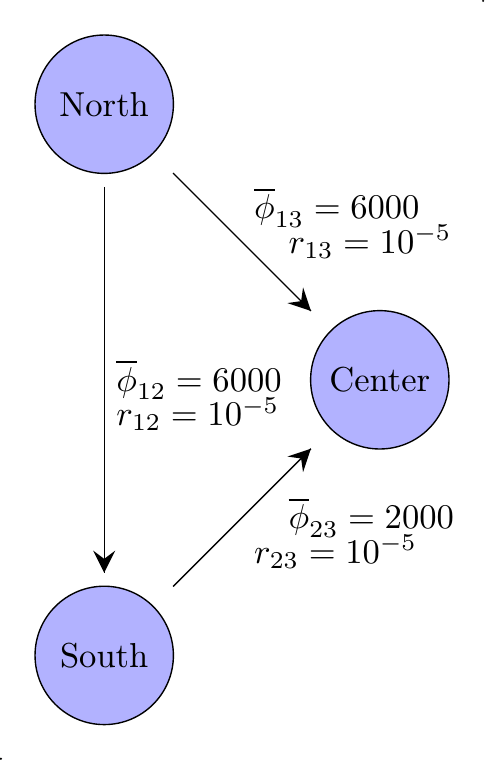} 	
\caption{Characteristics of the network.}   \label{fig:network}
\end{center}
\end{figure}

We distinguish three technologies to produce power, namely coal, gas and solar. Each technology is assumed to pollute the same and to have the same production costs, independently of the node at which it is used. The costs of producing with the coal, gas and solar technologies, in dollars per MWh, are 40, 80 and 0 respectively. The pollution emissions of the coal, gas and solar technologies, in tons of $\mathrm{CO}_2$ per MWh, are 1, 0.5 and 0 respectively. 

\begin{table}[H] \label{table-poll-cost}
\centering
\begin{tabular}{|c|c|c|}
\hline
  Technology    & \multicolumn{1}{l|}{Cost (dollars per MWh)} & \multicolumn{1}{l|}{Pollution (tons of $\mathrm{CO}_2$ per MWh)} \\ \hline
solar & 0                                           & 0                                                    \\ \hline
coal  & 40                                          & 1                                                    \\ \hline
gas   & 80                                          & 0.5                                                  \\ \hline
\end{tabular}
\caption{Cost and pollution of the technologies.}
\end{table}

 The difference between the producers is the availability of each of the technologies and the capacities per technology. To determine the cost and pollution functions of the providers, we assume they produce power with the cheapest technology available until it is saturated, time at which the cheapest of the remaining technologies is used. As mentioned in \Cref{rm-piecewise}, this results in piecewise linear functions.

\medskip
The first firm (North) produces by using the technologies (solar, coal, gas) in the proportions (0.3, 0.3, 0.4) and presents therefore the following cost function (in dollars, with $q$ in MWh) and pollution function (in tons of $\mathrm{CO}_2$, with $q$ in MWh)
\begin{align*}
c_1(q) &= \begin{cases}
0, & q\leq 1800, \\
40(q-1800), & q\in[1800,3600], \\
80(q-3600) + 72000, & q\in[3600,6000],
\end{cases}\; 
p_1(q) = \begin{cases}
0, & q\leq 1800, \\
(q-1800), & q\in[1800,3600], \\
0.5(q-3600) + 1800, & q\in[3600,6000].
\end{cases}
\end{align*}
The second firm (South) produces by using the technologies (solar, coal, gas) in the proportions (0.1, 0.4, 0.5), with cost and pollution functions
\begin{align*}
c_2(q) &= \begin{cases}
0, & q\leq 200, \\
40(q-200), & q\in[200,1000], \\
80(q-1000) + 32000, & q\in[1000,2000],
\end{cases}\;
p_2(q) = \begin{cases}
0, & q\leq 200, \\
q-200, & q\in[200,1000], \\
0.5(q-1000) +  800, & q\in[1000,2000].
\end{cases}
\end{align*}

The third firm (Center) produces by using the technologies (solar, coal, gas) in the proportions (0.2, 0.1, 0.7), with cost and pollution functions
\begin{align*}
c_3(q) &= \begin{cases}
0, & q\leq 2400, \\
40(q-2400), & q\in[2400,3600], \\
80(q-3600) + 48000, & q\in[3600,12000],
\end{cases}\;
p_3(q) = \begin{cases}
0, & q\leq 2400, \\
q-2400, & q\in[2400,3600], \\
0.5(q-3600) + 1200, & q\in[3600,12000].
\end{cases}
\end{align*}

We plot these piecewise linear functions in Figure \ref{fig:grid-cost-poll} below.
\begin{figure}[ht]
\centering
	\begin{tabular}{rl}
\includegraphics[scale=.17]{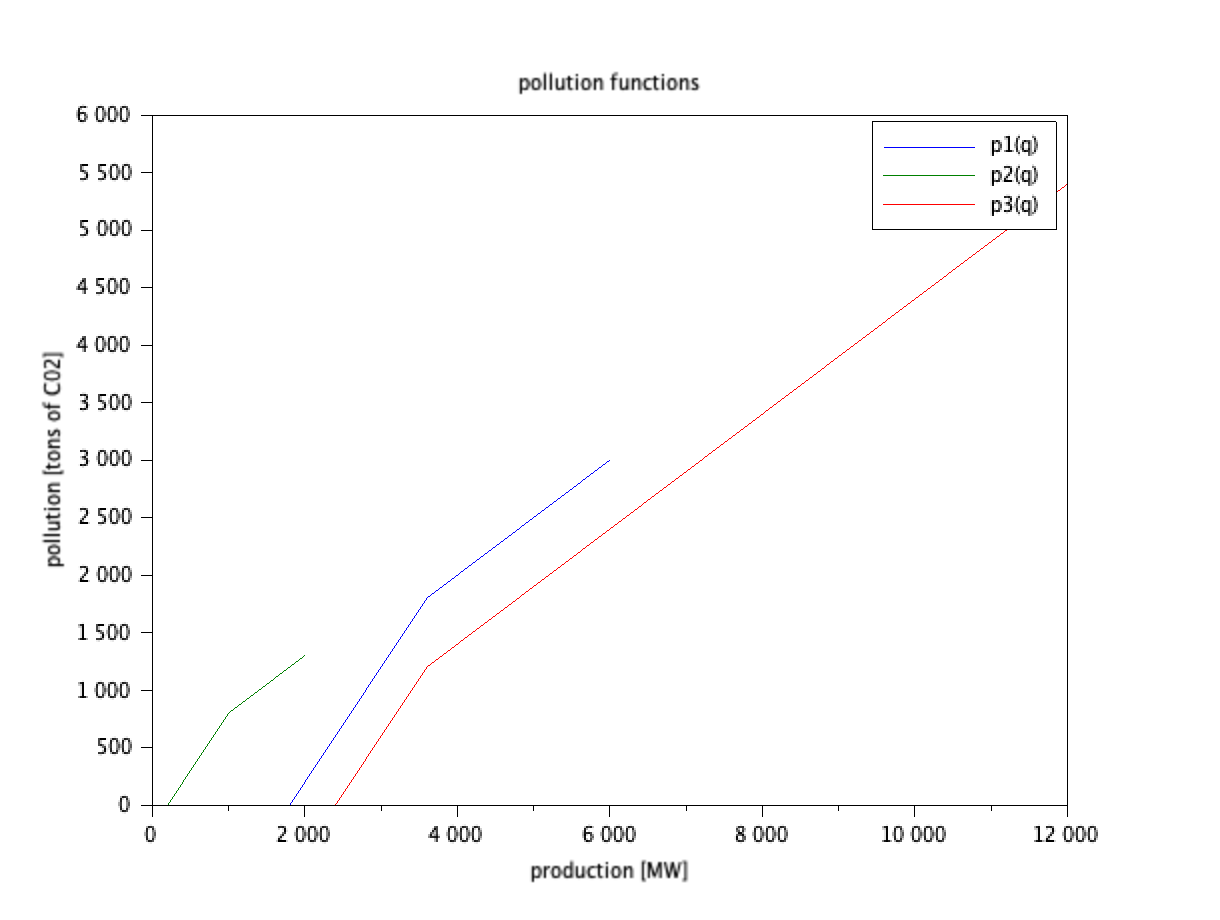}
		&
\includegraphics[scale=.17]{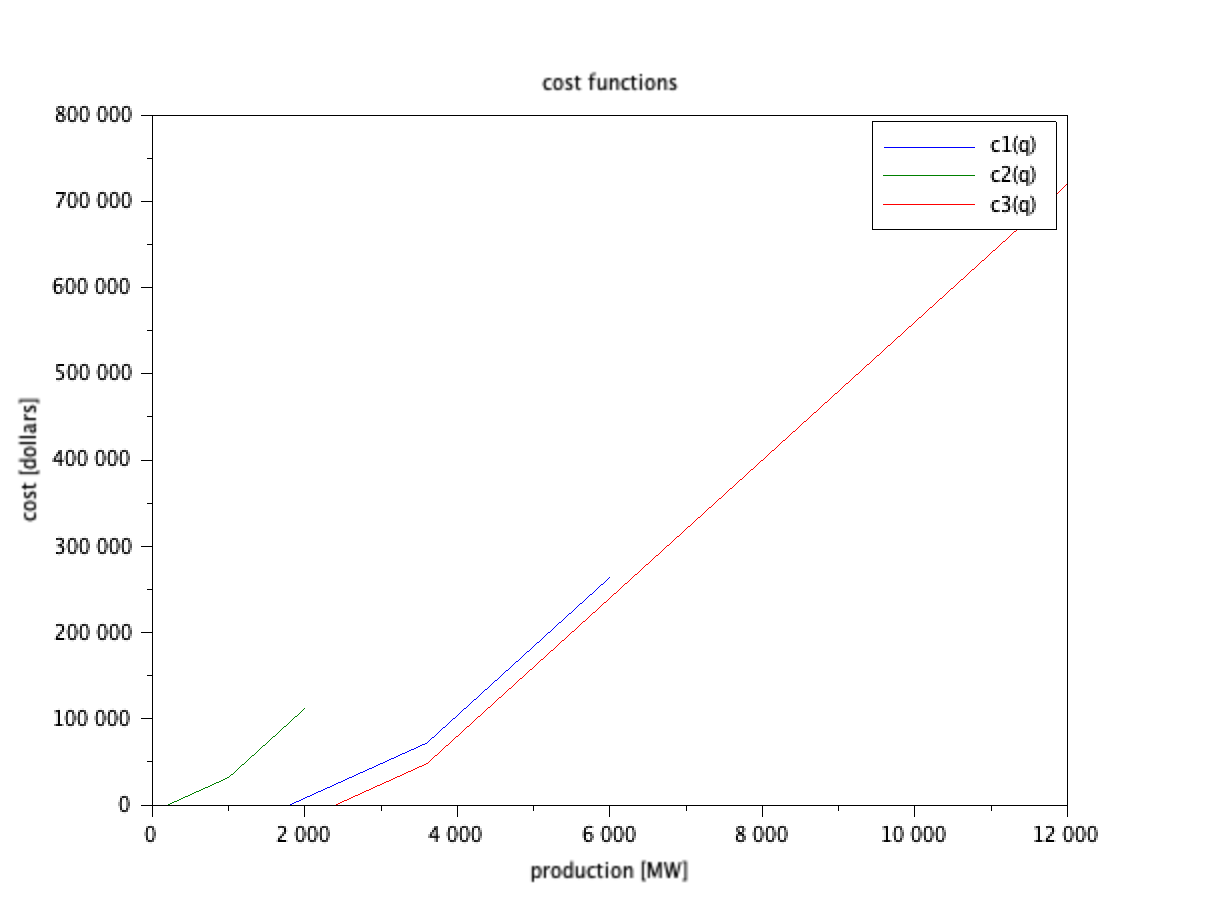}
	\end{tabular} 
\caption{Cost and pollution functions.}  \label{fig:grid-cost-poll}
\end{figure}

\subsection{No pollution regulation}

If ISO does not take into account pollution when assigning the production levels to the providers, production at the minimum cost is performed. Given the data, the production of each firm and the flow sent through each transmission line are given by
\[
q_t^1 = 3600,\; q_t^2 = 1000.4,\; q_t^3 = 5402, \; \phi^{e_{12}}_t =198 , \; \phi^{e_{13}}_t = 401 , \; \phi^{e_{23}}_t =198.
\]
 \Cref{fig:network-nopoll} shows the flows in the network in the absence of regulation.
\begin{figure}[H]
\centering
		\includegraphics[scale=.6]{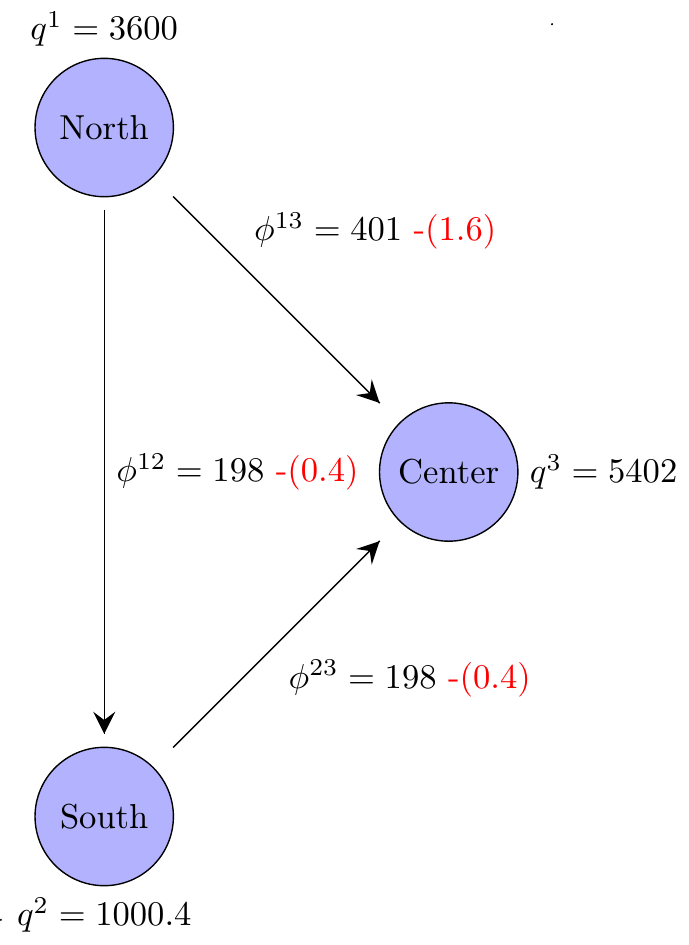} 	
\caption{Flow in the network without regulation.}  \label{fig:network-nopoll}
\end{figure}

We simulate a three months period with an initial pollution level of $L_0 = 8,000,000$ tons of ${\rm CO}_2$.  \Cref{fig:not-poll-mult} shows the pollution trajectories in this setting.
\begin{figure}[H]
\centering
	\begin{tabular}{rl}
\includegraphics[scale=.17]{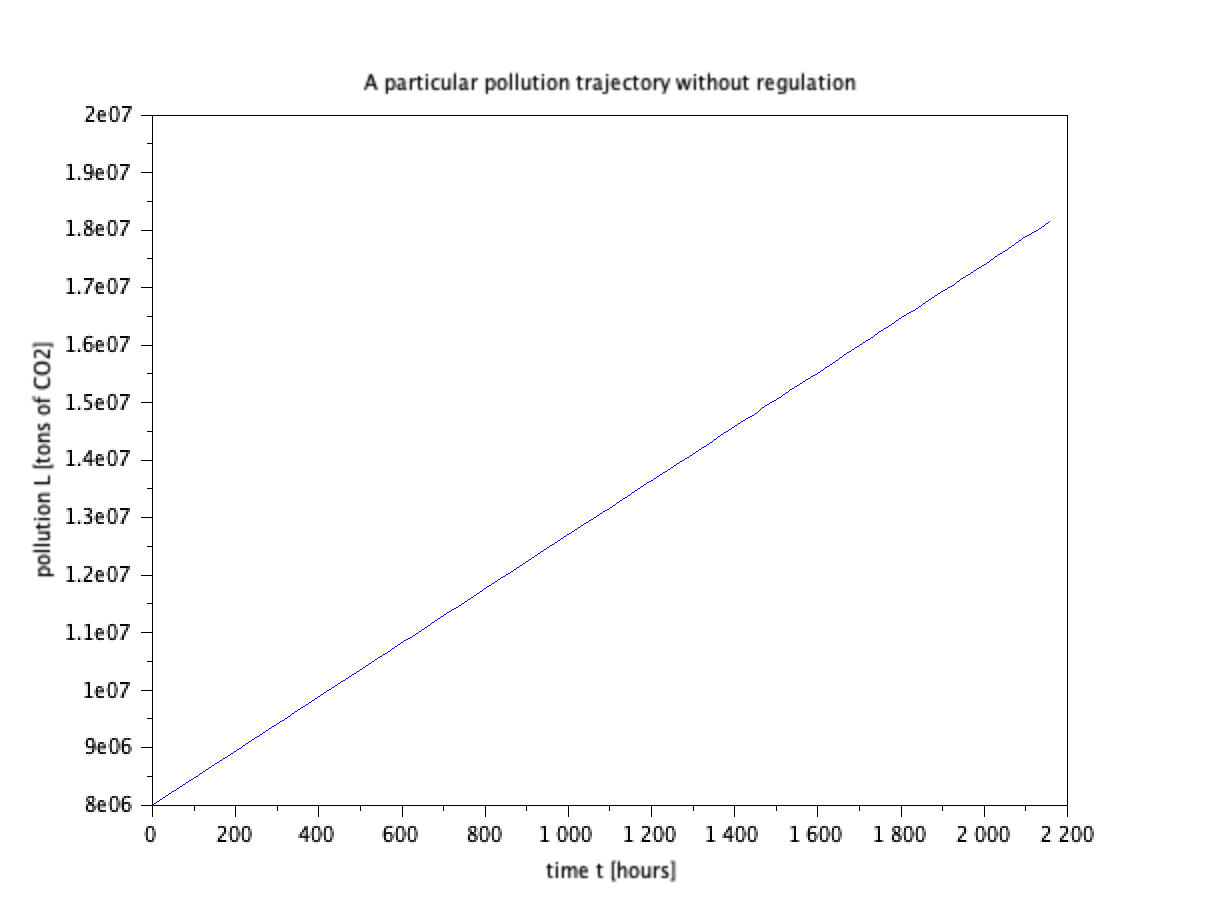}
		&
\includegraphics[scale=.17]{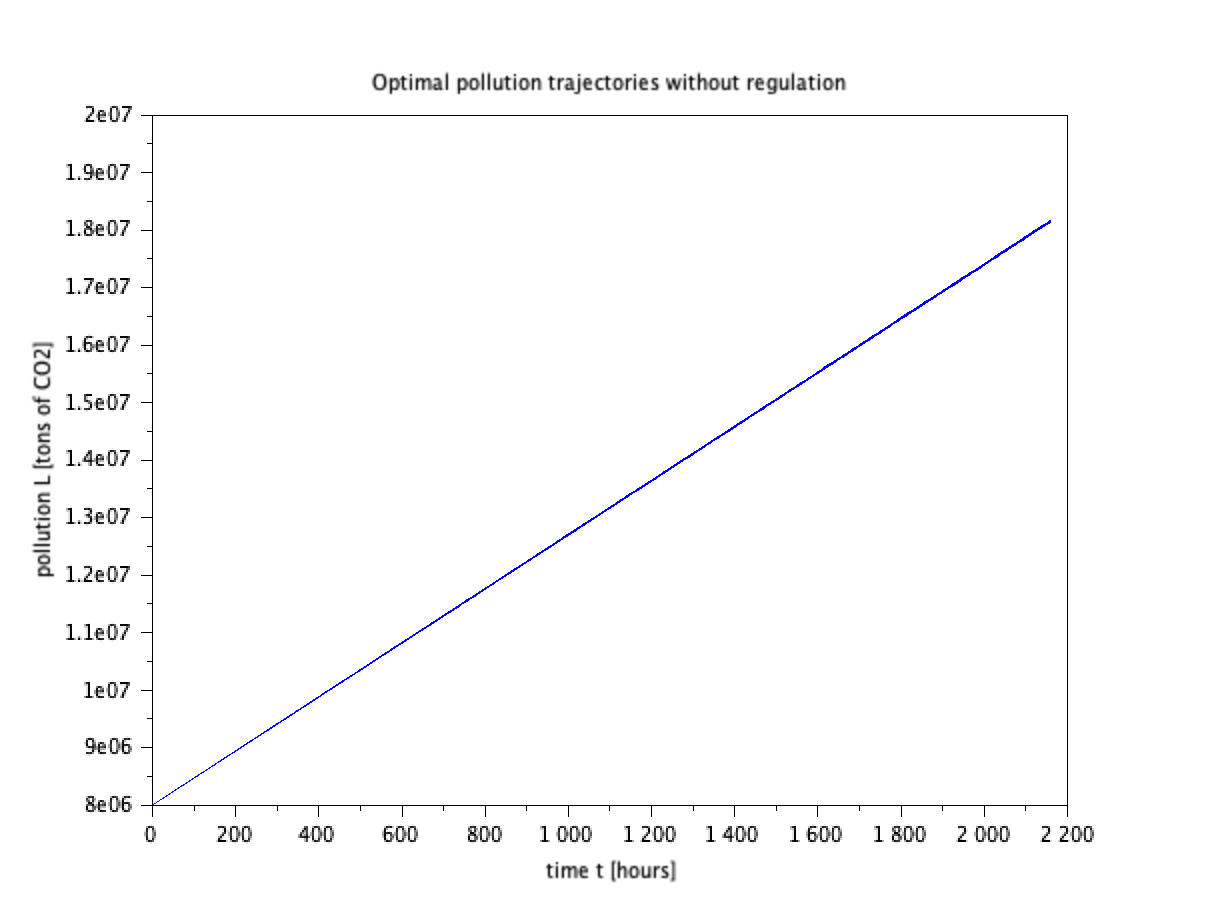}
	\end{tabular} 
\caption{Pollution without regulation.}  \label{fig:not-poll-mult}
\end{figure}

\subsection{Regulation}

With the same data as before, we consider a social cost function $\Lambda(x)=\lambda (x)_+$, with $\lambda=5$ in dollars per ton of $\mathrm{CO}_2$. We also set $\ell_0 = 10000000$ tons of $\mathrm{CO}_2$. The cost of the effort of reducing the pollution is given by the functions $h_i(a)=\frac{h_i}{2}a^2$, with $h_1=h_2=h_3=30000$ in dollars. The efforts take values in the sets $A_1=A_2=A_3=[0,0.3]$. The value function of the ISO in this case is shown in \Cref{fig:mult-vf}. 

\medskip
If we calculate the social cost of the pollution under no regulation (see \Cref{fig:not-poll-mult}), we obtain a value of $3.645 \times 10^{10}$. Under regulation we have $V_0=V(0,L_0)=1.44 \times 10^{10}$, so the social cost is reduced to less than the half. Let us mention that under no regulation the production costs are equal to $6.30 \times 10^8$. The difference between this number and the social cost under no regulation represents the compensation that a private ISO would have to receive, for instance from the government, to implement the regulation program in the case it does not value pollution. We obtain that this hypothetical compensation would be equal to  $3.582 \times 10^{10}$.

\begin{figure}[H]
\centering
\includegraphics[scale=.17]{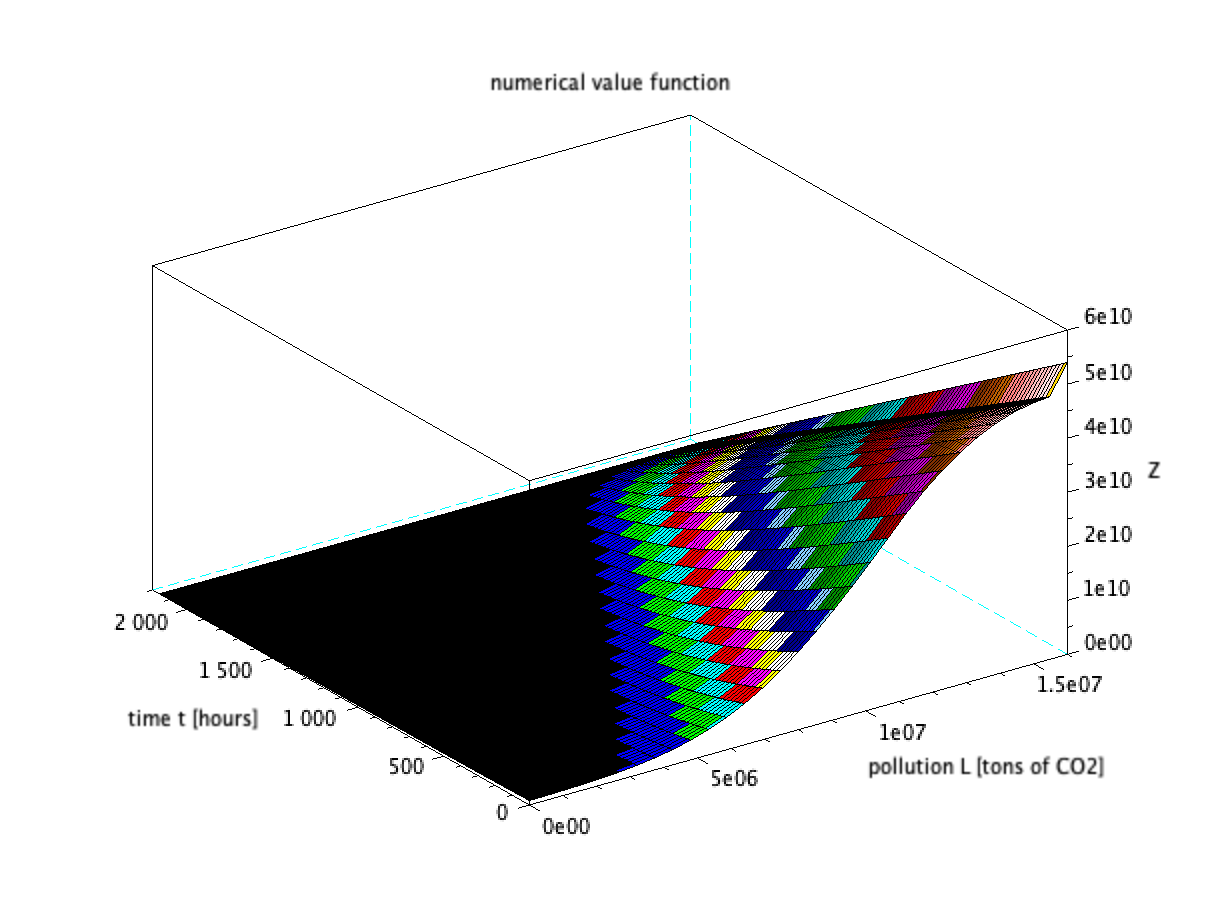}
\caption{Value function with regulation.}  \label{fig:mult-vf}
\end{figure}

\vspace{0.5em}
Figure \ref{fig:grid-p} shows the optimal pollution trajectories. We see that the contracts reduce the increment in the pollution levels by more than thirty percent (without regulation the increment is almost 10000000 tons of $\mathrm{CO}_2$ and with regulation the increment is around 6500000 tons of $\mathrm{CO}_2$). 
\begin{figure}[ht!]
\centering
	\begin{tabular}{cc}
				\includegraphics[scale=0.17]{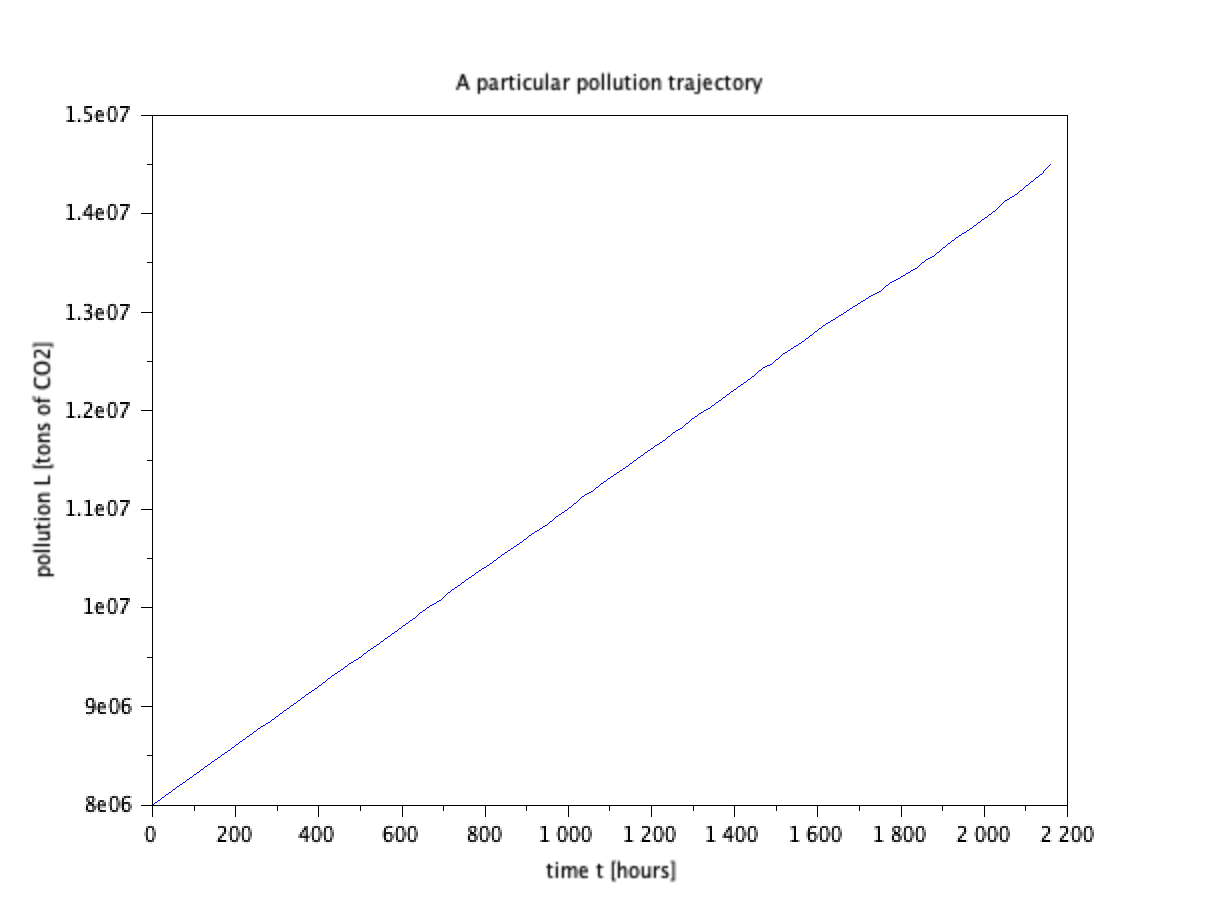} 	
&
				\includegraphics[scale=0.17]{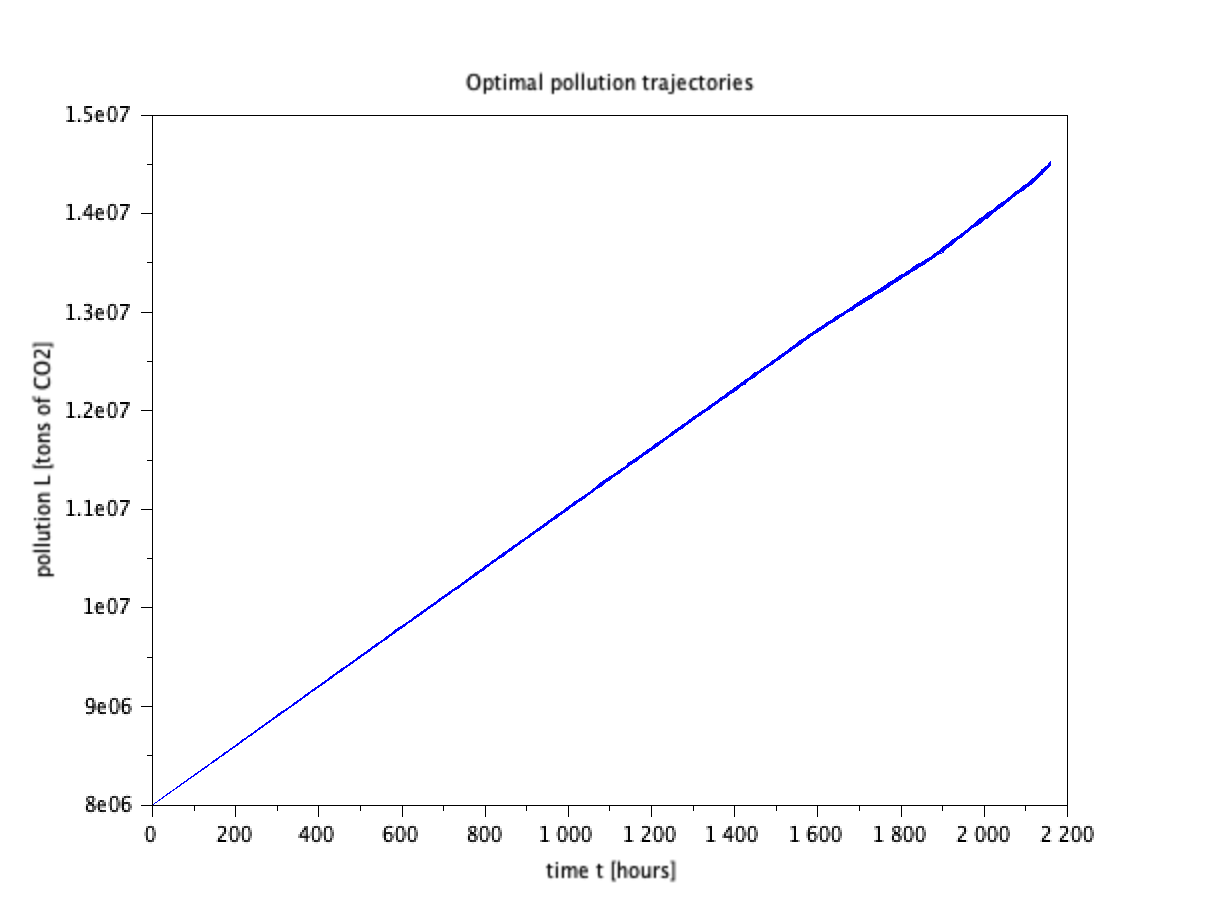} 	
	\end{tabular} 
\caption{Pollution trajectories under regulation.}  \label{fig:grid-p}
\end{figure}

The optimal production and flow trajectories are shown in Figure \ref{fig:mult-q-f}. The production asked to each firm is piece-wise constant, as well as the flows through the line, and we identify three periods of time in which we describe the actions of the providers. In the first period, the node in the Center is asked to produce exactly $2400$ MWh, which it does by using only the solar technology and therefore it does not pollute nor incur in any cost. Since this amount does not cover the local demand, the remaining of the demand in the Center is provided by the nodes in the North and the South. In these nodes all the available technologies are used to cover their local demands and what is required by the node in the Center. The node in the North is asked the highest production in the network of $5600$ MWh. The node in the South produces $2000$ MWh and receives a flow of $700$ MWh from the North. After some time has passed we move into the second period, in which the main change is that the node in the Center is asked to cover most of its local demand, by using all the available technologies and producing $5500$ MWh. The node in the North contributes with a flow of $500$ MWh and no flow is received from the node in the South. 
The second period occurs when the increase of pollution has been controlled for enough time and it becomes less costly to allow the node in the Center to produce power with the polluting technologies. This is a temporary situation and after some time the node in the Center is asked to act again as in the first period. In the third period, the node in the Center produces exactly $2400$ MWh by using only the solar technology. The situation here is very similar to the one in the first period, with slight differences in the production of the nodes in the North and the South, as well as the flows they send through the network. 

\begin{figure}[H]
\centering
	\begin{tabular}{cc}
				\includegraphics[scale=0.17]{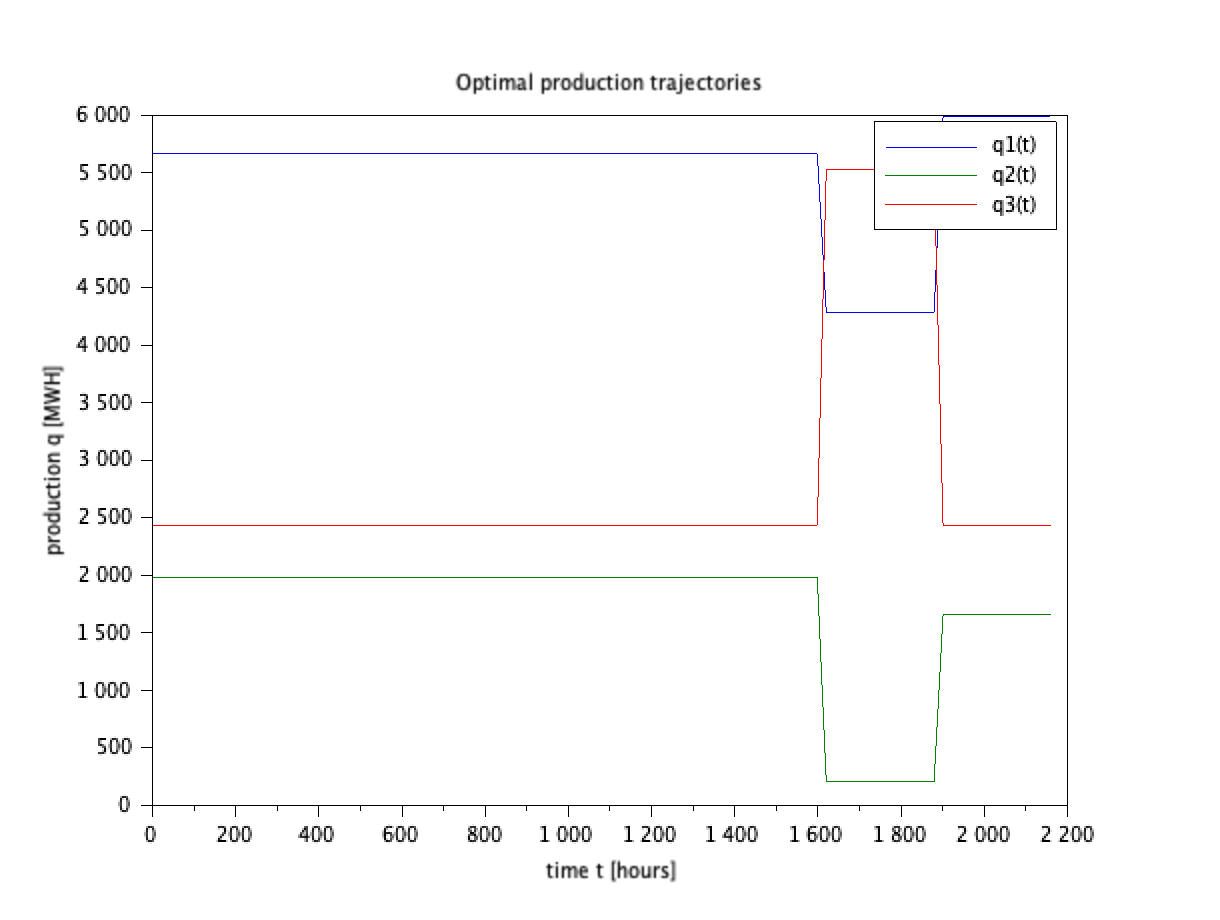} 	
&
				\includegraphics[scale=0.17]{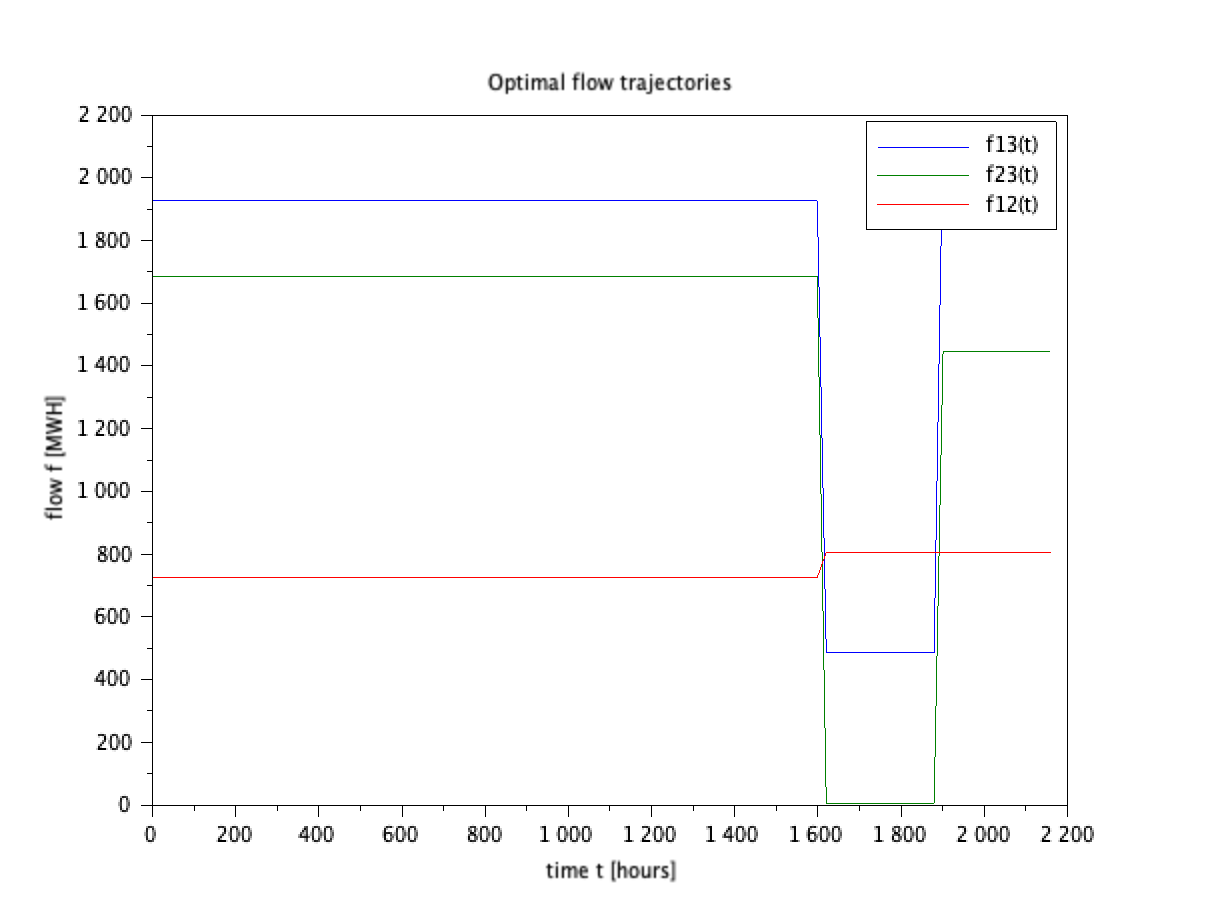} 	
	\end{tabular} 
\caption{Production and transmission plans under regulation.}  \label{fig:mult-q-f}
\end{figure}

The optimal efforts and the certainty equivalents of the producers are shown in \Cref{fig:mult-e-ce}. The efforts performed by the firms drop to zero at the end of the contract since, as shown in the example of Section \ref{sec-example}, when closer to the end of the contract it is too expensive for the ISO to encourage the firms. Prior to the drop, the node in the North, in charge of the highest production, is given incentives to perform full effort and its certainty equivalent increases linearly. The node in the Center receives full incentives during the second period in which it produces with all the technologies, and none during the first and third periods when it produces only with solar technology. The node in the South is the most expensive to encourage (its production is gas-based) and we see that during the third period it is given less incentives than the node in the North which results in an intermediate effort which decreases sooner to zero.
\begin{figure}[ht!]
\centering
	\begin{tabular}{cc}
		\includegraphics[scale=0.17]{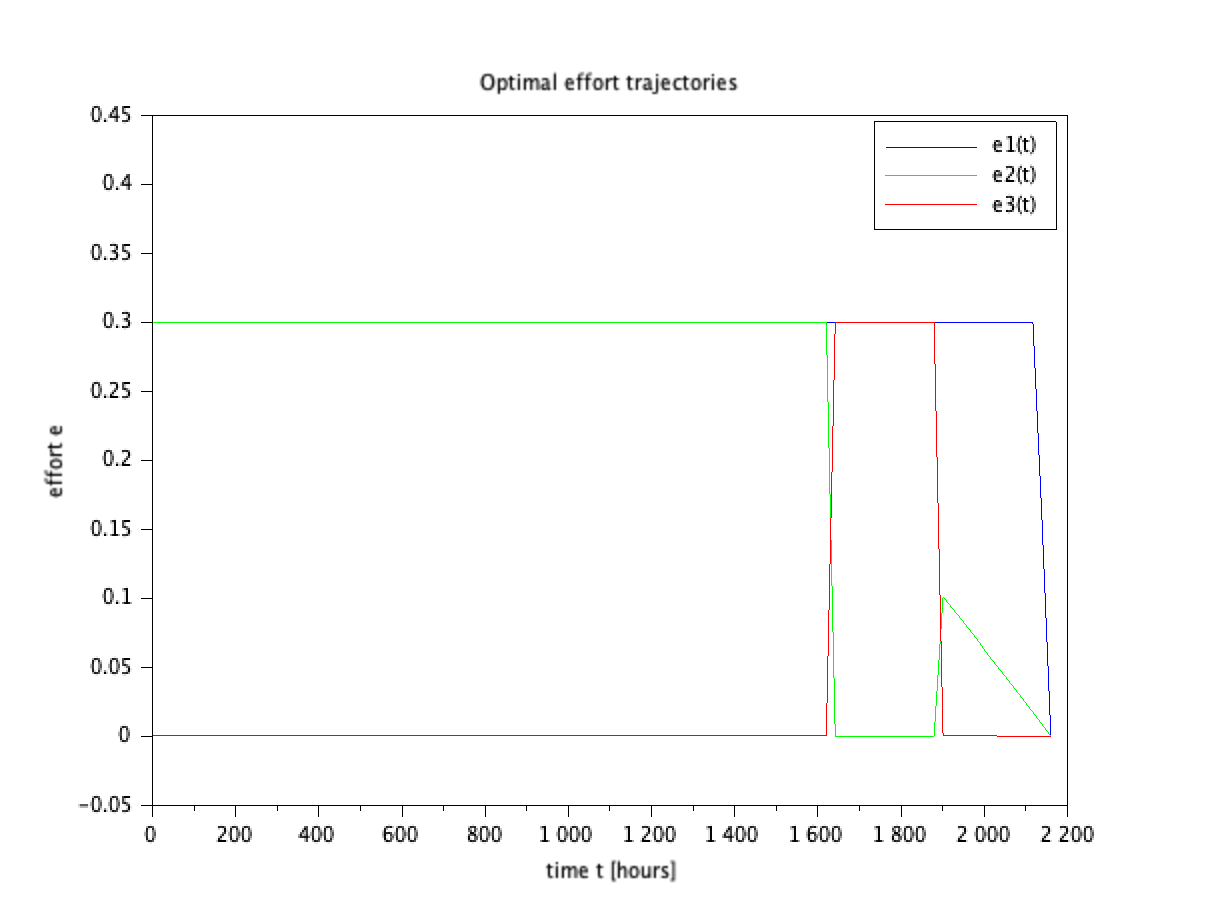} 	
		&
		\includegraphics[scale=0.17]{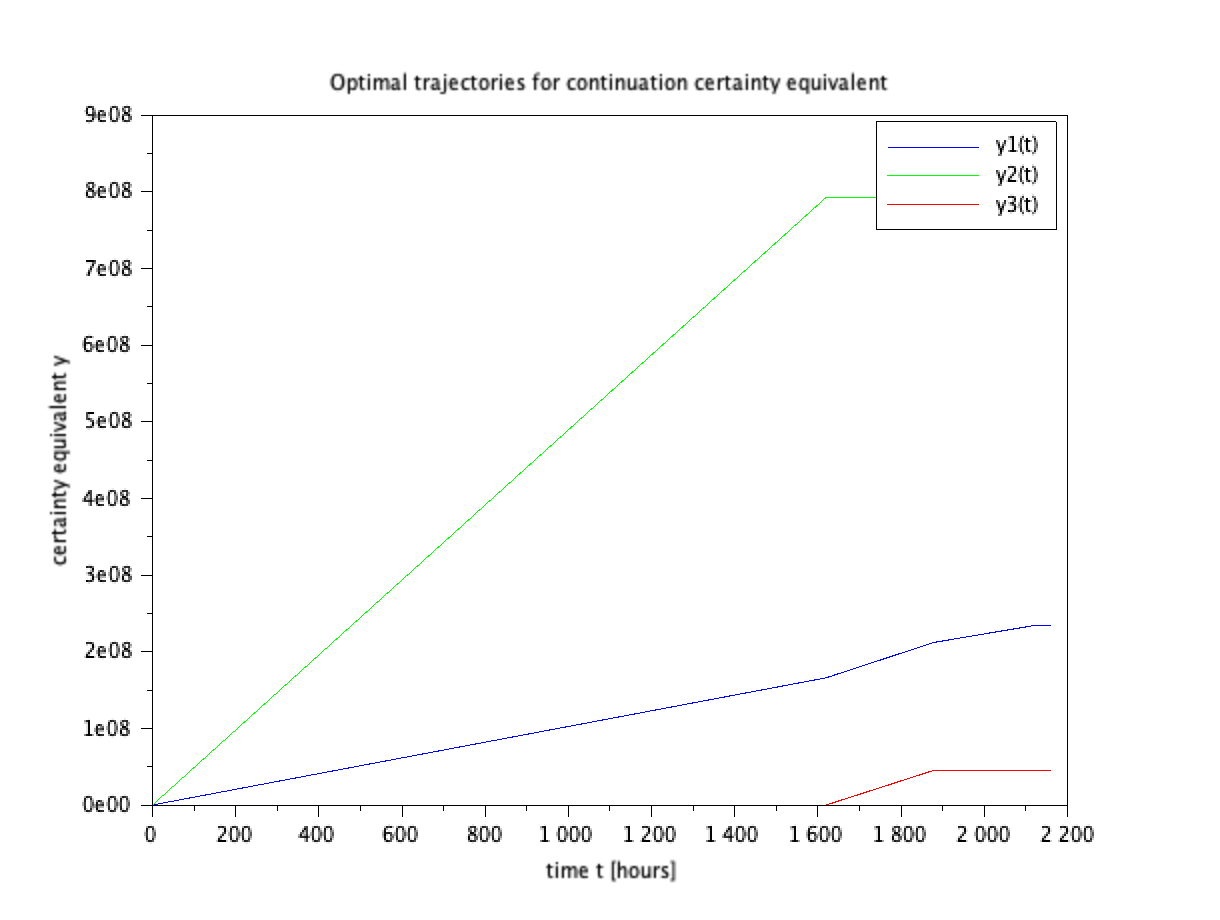} 	
	\end{tabular} 
\caption{Efforts and certainty equivalents with regulation.}  \label{fig:mult-e-ce}
\end{figure}

To conclude, let us discuss on the nature of the curves just presented. As mentioned in \Cref{sec-example}, constant trajectories are obtained when the spatial derivative of the value function is sufficiently big, which results in the (quadratic) optimisation over the production/flows being attained at the boundaries of some adjusted domain. In the setting with multiple technologies, the objective functional is defined piece-wise and these boundaries depend on the derivative of the value function. The plan of the ISO is dynamically updated and we can identify periods of time during which it remains constant. However, seen as whole trajectories, the optimal production and transmission plans are not constant, which means that the solution to problem $V_0$ of the ISO does not coincide with the one of the simpler problem $V^d$. Indeed, we have $V^d=1.75 \times 10^{10}$ and the solution to problem $V^d$ is attained at the point
\[
q^1=5829.6,\; q^2=1808,\; q^3=2436,\; \phi^{1,2}= 400,\; \phi^{1,3}=2400,\; \phi^{2,3}=1200.
\]
To conclude, we perform a sensitivity analysis on the parameter $\sigma$, corresponding to the volatility of the total pollution and thus representing the strength of moral hazard in the problem. We set $\sigma_0=200$ in tons of $\mathrm{CO}_2$ and we display in \Cref{fig:sigma} the optimal production and effort of the first provider for different values of $\sigma$. The behaviour of the second and third producers can be deduced from the pictures, by taking into account the discussion above. 
\begin{figure}[ht!]
\centering
	\begin{tabular}{cc}
		\includegraphics[scale=0.14]{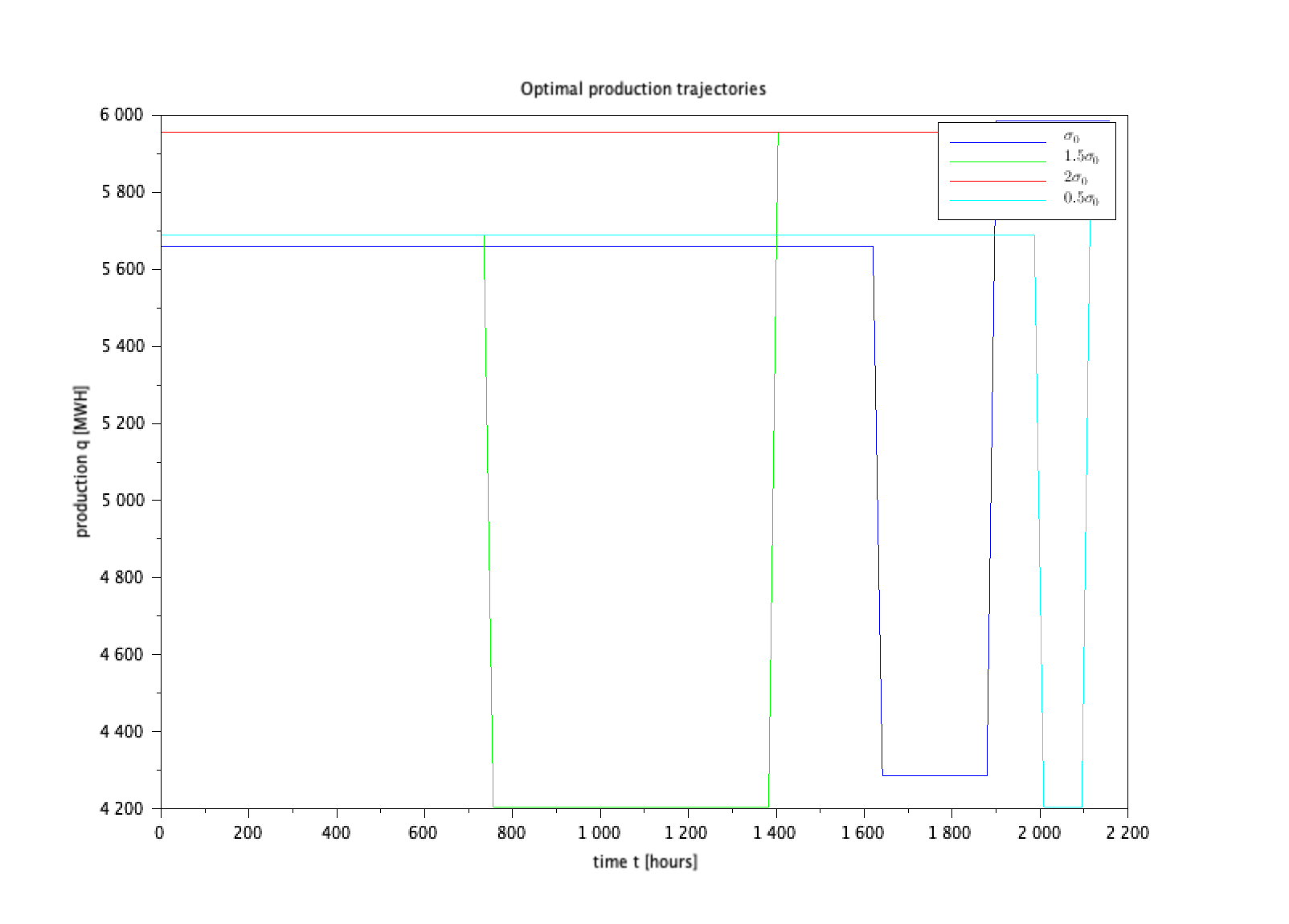} 	
		&
		\includegraphics[scale=0.17]{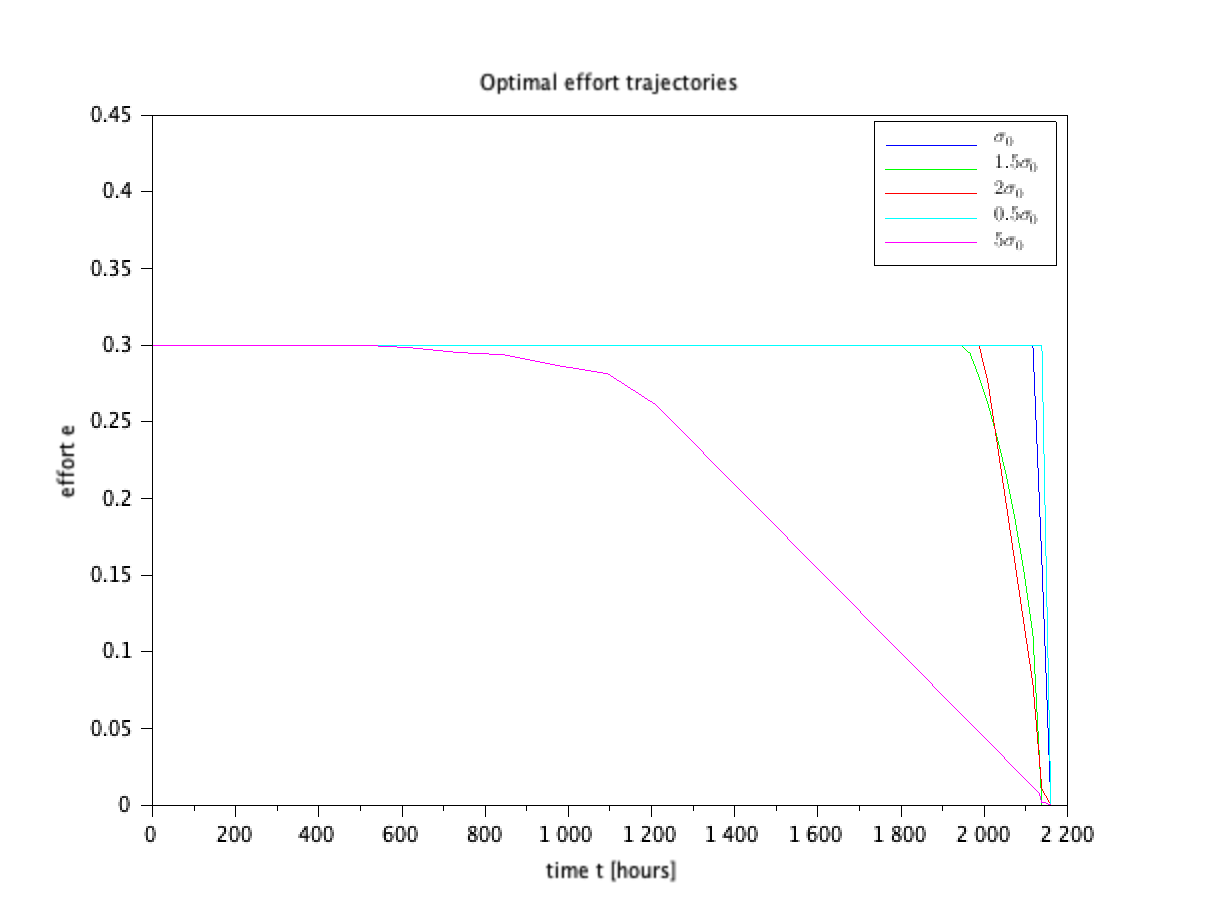} 	
	\end{tabular} 
\caption{Efforts and certainty equivalents with regulation.}  \label{fig:sigma}
\end{figure}

We see that the length of the three periods of production is affected by the values of $\sigma$. The smaller the value of this parameter, the shorter is the length of the second period, which is the only period in which the third provider is allowed to pollute. We interpret this by the fact that when decreasing $\sigma$, it is easier to detect moral hazard and to implement a strict policy with almost no pollution. Moreover, when $\sigma$ increases, the change from the first to the second period occurs faster and the third configuration becomes predominant. For big values of $\sigma$, in this setting greater than $2\sigma_0$, it is optimal to produce according to the third configuration for the whole life of the contracts. This means that problems $V_0$ and $V^d$ share the same solution. Concerning the optimal efforts, we see that they decrease to zero when the volatility is bigger, since the actions of the providers are harder to detect. In the limit case, when $\sigma$ is very large and moral hazard is too strong, the contracts cannot help to control pollution since the providers will exert practically no effort.

{\small
\bibliography{bibliographyDylan}}

\begin{appendix}

\section{The model} \label{ap:model}


We start by listing the properties that the functions in our model satisfy.

\begin{assumption} $(i)$ For every $i\in\{1,\dots,N\}$ we assume

\smallskip
$\bullet$ the function $c_i:\R_+\longrightarrow\R_+$ is continuous and increasing$;$

\smallskip
$\bullet$ the function $p_i:\R_+\longrightarrow\R_+$ is continuous and increasing$;$

\smallskip
$\bullet$ the function $h_i:A_i \rightarrow \R_+$ is continuous, increasing and strictly convex$;$

\medskip
$(ii)$ the function $\Lambda:\R\longrightarrow\R_+$ is continuously differentiable with quadratic growth, that is, there exists $K_\Lambda>0$ such that $|\Lambda(x)|\leq K_\Lambda(1+|x|^2)$ for every $x\in\R$.

\end{assumption}

We define also the following constants that will be used in some of the proofs
\begin{equation}\label{eq:constants}
\bar{c} := \max_{i\in\{1,\dots,N\}} \bigg\{  \sup_{q\in[0,Q^i]} c_i(q) \bigg\}, \; \bar{p} := \max_{i\in\{1,\dots,N\}} \bigg\{ \sup_{q\in[0,Q^i]} p_i(q)\bigg\}, \;\bar{h} := \max_{i\in\{1,\dots,N\}} \bigg\{ \sup_{a\in A_i} h_i(a)\bigg\}.
\end{equation}


We present now the weak formulation of the problem. We place ourselves in a probability space $(\Omega,\Fc,\P)$, representing the randomness in the weather conditions that cannot be controlled. The noise of the model will be given by a one-dimensional standard Brownian motion $W$ and the volatility $\sigma>0$. Let us define the driftless Brownian motion 
\[
\mathrm{d}L_t = \sigma \mathrm{d}W_t.
\]
We denote by $\mathbb{F}:=(\Fc_t)_{t\in[0,T]}$, the filtration generated by $L$, completed under the measure $\mathbb{P}$. Each producer $i\in\{1,\dots,N\}$, has a set of actions $A_i\subset[0,1]$,  a closed interval, to reduce its own pollution. We denote by $\Ac_i$ the space of controls of producer $i$, that is $a^i\in\Ac_i$ if it is an $A_i$-valued, $\mathbb{F}$-predictable process. We define the set of joint actions $\Ac=\prod_{i=1}^N \Ac^i$ and for a joint effort $a=(a^1,\dots,a^N)\in\Ac$ the probability measure
\[
\frac{\mathrm{d}\P^a}{\mathrm{d}\P} =\Ec\bigg( \sum_{i=1}^N \int_0^T \frac{(1-a^i_t)p_i(q^i_t)}{\sigma} \mathrm{d}W_t \bigg).
\]
Notice that since $a$ is bounded, and by Novikov's condition, the stochastic exponential is indeed a true martingale. We have thus, by Girsanov's theorem, that $W^a_t=W_t - \int_0^t \sum_{i=1}^N \frac{(1-a^i_s)p_i(q^i_s)}{\sigma} \mathrm{d}s$ is a $\P^a$--Brownian motion and the pollution process has the desired dynamics
\[
 \mathrm{d}L_t = \sum_{i=1}^N (1-a^i_t)p_i(q_t^i) \mathrm{d}t + \sigma \mathrm{d}W^a_t.
\]
Finally, for $i\in\{1,\dots,N\}$, we define the set of actions of the other producers $\Ac^{-i} = \prod_{j\neq i} \Ac^i$.

\begin{remark}
As the reader may have noted, both $\P^a$ and $W^a$ depend on the production plan. To ease the notation, specially when solving the problem of the providers for whom $(q,\phi)\in\Pc$ is fixed, we refrain from adding an additional index to the probability measures and their Brownian motions.
\end{remark}

\section{Reduction of the class of contracts} \label{ap:main}

\subsection{BSDE characterisation of Nash equilibria}

In this section, for any admissible contract $(q,\phi,\xi)\in\Cc$ offered by the ISO, we characterise the set of Nash equilibria $\text{NE}(q,\phi,\xi)$ through the solutions to a multidimensional BSDE. We start with an auxiliary result. All the proofs are deferred to \Cref{ap:proofs}.
\begin{lemma}\label{lemma:a-star}
For every $i\in\{1,\dots,n\}$, the correspondence $a^{\star,i}: \R_+^N\times \R^N \rightrightarrows A_i$ given by
\begin{equation}\label{eq:a-star-appendix}
a^{\star,i} (q,z) := \argmin_{a\in A_i} \big\{ h_i(a)-z^i (1-a)p_i(q^i) \big\},
\end{equation}
is single-valued and continuous. 
\end{lemma}

Suppose a contract $(q,\phi,\xi)\in\Cc$ is given. Based on the previous lemma, we introduce the following $N$-dimensional BSDE
\begin{equation}\label{eq:BSDE}
Y_t = \xi - \int_t^T f(q_s,Z_s) \mathrm{d}s -\int_t^T Z_s \sigma \mathrm{d}W_s,
\end{equation}
where the vector function $f:\R_+^N \times\R^N\longrightarrow\R^N$ is defined by
\begin{equation}\label{eq:generator}
f^i(q,z) := h_i\big(a^{\star,i}(q,z)\big) + c_i(q^i)-z^i \sum_{j=1}^N \big(1-a^{\star,j}(q,z)\big)p_j(q^j) + \frac{\rho\sigma^2}{2} (z^i)^2 ,\; i\in\{1,\dots,N\}.
\end{equation}

We define now the notion of solution to the previous BSDE.
\begin{definition}
The space $\mathbb{H}^2_\text{\rm loc}(\mathbb{R}^N)$ consists in the $\R^N$-valued predictable processes $Z$ satisfying 
\[
\mathbb{P}\bigg[\int_0^T |Z_s^i|^2 \mathrm{d}s<\infty \bigg]=1, \; \forall i\in\{1,\dots,N\}.
\]
\end{definition}

\begin{definition}
A solution to {\rm BSDE} \eqref{eq:BSDE} is a pair $(Y,Z)$ such that $Y$ is a continuous, $\R^N$-valued, $\F$-adapted process, $Z\in \mathbb{H}^2_\text{\rm loc}(\mathbb{R}^N)$, which satisfies \eqref{eq:BSDE}.
\end{definition}

Before presenting the main result of the section, we need to introduce a family of processes that will be used to reformulate the class of admissible remunerations, as we show below. For $y\in\R^N$ and $Z\in\mathbb{H}^2_\text{loc}(\R^N)$, we define the process $Y^{y,q,Z}$ through the following SDE 
\[
Y_t^{y,q,Z} := y + \int_0^t f(q_s,Z_s) \mathrm{d}s + \int_0^t Z_s \sigma \mathrm{d}W_s,\; t\in[0,T].
\]

We define finally the following class of processes, using the notation $a^{\star,i}_s := a^{\star,i}(q_s,Z_s)$ and with $\Tc_{0,T}$ the set of stopping times with values in $[0,T]$
\begin{equation}\label{eq:classZ}
\Zc = \bigg\{ Z\in \mathbb{H}^2_\text{loc}(\mathbb{R}^N) :  \sup_{\tau \in\Tc_{0,T}} \E\Big[ \big|\Uc_A(Y_\tau^{0,q,Z,i})\big|^p \Big] < \infty, \; \text{for}\; \text{some}\; p>1   \bigg\}.
\end{equation}

\begin{remark}
The definition of the class $\Zc$ is based on the production plan fixed at beginning of this section. However, it is easy to see that the class is independent of the contract $(q,\phi,\xi)\in\Cc$ due to the multiplicative property of the function $\Uc_A$ and the boundedness of the plan $(q,\phi)\in\Pc$. For this reason we do not index the class on the underlying contract.
\end{remark}

The following proposition, similar to the result established in \cite{elie2019contracting}, extends the link between principal--agent problems and BSDEs. It is well-known that in the case of a single agent, his problem can be solved by looking at a one-dimensional BSDE. In our setting with $N$ agents, we need to consider the $N$-dimensional BSDE defined above.

\begin{proposition} \label{prop:optimal-effort}
Fix a contract $(q,\phi,\xi)\in\Cc$. For any Nash equilibria $a^\star\in\text{\rm NE}(q,\phi,\xi)$ there exists a solution $(Y,Z)$ to {\rm BSDE} \eqref{eq:BSDE} such that $Z\in\Zc$ and we have, for every $i\in\{1,\dots,N\}$, $\mathrm{d}t\otimes\mathrm{d}\P$\text{\rm--a.e.} on $[0,T]\times\Omega$
\begin{equation}\label{eq:optimal-effort}
a^{\star,i}_s := a^{\star,i}(q_s,Z_s).
\end{equation}
Conversely, let $(Y,Z)$ be a solution to {\rm BSDE} \eqref{eq:BSDE} such that $Z\in\Zc$. Then, the joint action $a^\star$ defined by \eqref{eq:optimal-effort} belongs to $\text{\rm NE}(q,\phi,\xi)$.
\end{proposition}

The previous proposition allows us to reformulate the problem of the ISO \eqref{eq:ISO-problem} as a standard stochastic control problem. Indeed, we can represent the admissible remunerations as the terminal values of processes of the form $Y^{y,q,Z}$. Therefore, we have the equality
\begin{equation}\label{contracts-reformulation}
\Cc = \big\{ (q,\phi,Y_T^{y,q,Z}): (q,\phi)\in\Pc,\; y\in\R^n,\; Z\in\Zc\; \text{\rm with}\; \Uc_A(y^i)\geq R_0^i,\; \forall i\in\{1,\dots,N\}  \big\},
\end{equation}
which leads to the reformulation of problem \eqref{eq:ISO-problem-2}, stated in \Cref{th:main}. The controlled process $Y^{y,q,Z}$ allows to tackle the problem of the ISO by following the dynamic programming approach developed in \cite{cvitanic2018dynamic}. It plays the same role as the continuation utility process introduced in \cite{sannikov2008continuous}, with the difference that in our model it represents the certainty equivalent of the producers. The last result of this section illustrates this point and states that when a contract having the form \eqref{contracts-reformulation} is offered, the unique Nash equilibrium is given by the process \eqref{eq:optimal-effort}. 

\begin{proposition} \label{prop:contracts}
For $(q,\phi,y,Z)\in\Pc\times\R^n\times\Zc$, consider the contract $(q,\phi,\xi)\in\Cc$ with $\xi=Y_T^{y,q,Z}$. Then the joint effort $a^\star$, given by \eqref{eq:optimal-effort}, is the unique element in $\text{\rm NE}(q,\phi,\xi)$ and  $V_0^i(a^{\star,-i},\xi^i,q)=\Uc_A(y^i)$.
\end{proposition}

\subsection{Proofs}\label{ap:proofs}

\begin{proof}[Proof of {\rm Lemma \ref{lemma:a-star}}]
Since the set $A_i$ is compact and the maps $h_i$ and $p_i$ are continuous, the map 
\[
a \longmapsto h_i(a) - z^i (1-a)p_i(q^i),
\]
has a minimiser over $A_i$. The uniqueness of the minimiser follows from the strict convexity of $h_i$. From the maximum Theorem (see for instance \cite[Theorem 17.31]{aliprantis2006infinite}) the function $a^{\star,i}$ is continuous.
\end{proof}

\begin{proof}[Proof of {\rm Proposition \ref{prop:optimal-effort}}]

{\color{red}(i)} Let $a^\star\in \text{\rm NE}(q,\phi,\xi)$ be a Nash equilibrium for the contract $(q,\phi,\xi)\in\Cc$. Then for each player $i\in\{i,\dots,N\}$ the action $a^{\star,i}$ maximises his expected utility, given the actions of the others $a^{\star,-i}$
\[
U_0^i(a^{\star,i},a^{\star,-i},\xi^i,q) = V_0^i(a^{\star,-i},\xi^i,q)  = \sup_{a\in\Ac_i}   U_0^i(a,a^{\star,-i},\xi^i,q).
\]
We fix $a^{\star,-i}$ and define the following family of random variables, for $\tau\in\Tc_{0,T}$
\[
V^i(\tau,a^{\star,-i},\xi^i,q):= \esssup_{a^i \in\Ac_i} \E^{\P^{a^i \otimes_i a^{\star,-i}}} \bigg[  \Uc_A \bigg( \xi^i - \int_\tau^T \big( h_i(a_s^i) + c_i(q_s^i) \big) \mathrm{d}s \bigg)  \bigg | \Fc_\tau\bigg], \;\P\text{\rm--a.s.}
\]
Let us discuss the integrability of this family, which is inherited from the integrability of $\xi\in\Rc$ because the processes $a^i$ and $q^i$ are bounded. From the definition of the set $\Ac_i$, it is clear that the family $(U^i(\tau,a^i,a^{\star,-i},\xi^i,q))_{a^i \in\Ac_i} $, with
\[
U^i(\tau,a^i,a^{\star,-i},\xi^i,q)) :=  \E^{\P^{a^i \otimes_i a^{\star,-i}}} \bigg[  \Uc_A \bigg( \xi^i - \int_\tau^T \big( h_i(a_s^i) + c_i(q_s^i) \big) \mathrm{d}s \bigg)  \bigg | \Fc_\tau\bigg],
\]
is directed upwards. Therefore, there exists a sequence $(a^i_n)_{n\in\N}$ in $\Ac_i$ such that $(U^i(\tau,a^i_n,a^{\star,-i},\xi^i,q))_{n\in\N}$ is non-decreasing and $V^i(\tau,a^{\star,-i},\xi^i,q)=\lim_{n\to\infty} U^i(\tau,a^i_n,a^{\star,-i},\xi^i,q))$. From the monotone convergence theorem we have for every $\hat p >1$
\[
\E \big[ |V^i(\tau,a^{\star,-i},\xi^i,q)|^{\hat{p}} \big] = \lim_{n\to\infty} \E \big[ |U^i(\tau,a^i_n,a^{\star,-i},\xi^i,q)) |^{\hat p} \big].
\]
Take $p>1$ such that $\E[|\Uc_A(\xi^i)|^p]<\infty$, then by H\"{o}lder's inequality and since the density measures have moments of any order we have for every $p_1\in(1,p)$ and $p_2\in(1,p_1)$ 
\begin{align*}
 \E \big[ |U^i(\tau,a^i_n,a^{\star,-i},\xi^i,q)) |^{p_2} \big] & \leq  \E^{\P^{n}} \big[ |U^i(\tau,a^i_n,a^{\star,-i},\xi^i,q)) |^{p_1} \big]^\frac{p_2}{p_1}  ~  \E^{\P^{n}}  \bigg[  \bigg( \frac{\textrm{d}\P}{\textrm{d}\P^n}  \bigg)^{p_3} \bigg]^\frac{1}{p_3}\\
 & \leq  \E^{\P^{n}}  \bigg[  \bigg| \Uc_A \bigg( \xi^i - \int_\tau^T \big( h_i(a_s^i) + c_i(q_s^i) \big) \mathrm{d}s \bigg) \bigg|^{p_1}  \bigg]^\frac{p_2}{p_1}  \mathrm{e}^{\frac{(p_3-1)TN^2\bar{p}^2}{2 \sigma^2}} \\
 & \leq \E^{\P^{n}}  \big[|\Uc_A(\xi^i)|^{p_1} \big]^\frac{p_2}{p_1}  \mathrm{e}^{\rho T(\bar{h} + \bar{c})p_2 }   \mathrm{e}^{\frac{(p_3-1)TN^2\bar{p}^2}{2 \sigma^2}} \\
 & \leq \E \big[|\Uc_A(\xi^i)|^{p} \big]^\frac{p_2}{p}    \E \bigg[  \bigg( \frac{\textrm{d}\P^n}{\textrm{d}\P}  \bigg)^{p_4} \bigg]^\frac{p_2}{p_4 p_1}   \mathrm{e}^{\rho T(\bar{h} + \bar{c})p_2 }  \mathrm{e}^{\frac{(p_3-1)TN^2\bar{p}^2}{2 \sigma^2}} \\
 & \leq  \E \big[|\Uc_A(\xi^i)|^{p} \big]^\frac{p_2}{p}     \mathrm{e}^{\frac{(p_4-1)TN^2\bar{p}^2 p_2}{2 \sigma^2 p_1}}     \mathrm{e}^{\rho T(\bar{h} + \bar{c})p_2 }     \mathrm{e}^{\frac{(p_3-1)TN^2\bar{p}^2}{2 \sigma^2}},
\end{align*}
where $\P^n$ denotes $\P^{a_n^i \otimes_i a^{\star,-i}}$, $\bar{h}$, $\bar{c}$ and $\bar{p}$ are defined in \eqref{eq:constants}, $p_3$ and $p_4$ are the H\"{o}lder's conjugates of $\frac{p_1}{p_2}$ and $\frac{p}{p_1}$ respectively. It follows that 
\begin{equation}\label{eq:uniformly-integrable}
\sup_{\tau \in\Tc_{0,T}} \E \big[|V^i(\tau,a^{\star,-i},\xi^i,q)|^{p_2}\big]<\infty.
\end{equation}

We have next the following dynamic programming principle (see for instance \cite[Theorem 3.4]{karoui2013capacities2})\footnote{The assumptions of the theorem are satisfied by the continuity of all the functions in our model (see \cite[Remark 3.10]{karoui2013capacities2}).}, for any $\theta\in\Tc_{0,T}$ such that $\tau\leq\theta$ it holds $\P$--a.s.
\[
V^i(\tau,a^{\star,-i},\xi^i,q) = \esssup_{a^i\in \Ac_i} \E^{\P^{a^i \otimes_i a^{\star,-i}}} \bigg[ V^i(\theta,a^{\star,-i},\xi^i,q)   \mathrm{e}^{\rho \int_\tau^\theta ( h_i(a_s^i) + c_i(q_s^i) ) \mathrm{d}s}  \bigg | \Fc_\tau\bigg].
\]
Then, for every $a^i\in\Ac^{-i}$ the family $\big(V^i(\tau,a^{\star,-i},\xi^i,q)  \mathrm{e}^{\rho \int_0^\tau ( h_i(a_s^i) + c_i(q_s^i) ) ds}\big)_{\tau\in\Tc_{0,T}}$ is a $\P^{a^i \otimes_i a^{\star,-i}}$--supermartingale\footnote{The family is integrable since $a^i$ and $q$ are bounded and \Cref{eq:uniformly-integrable} holds.} system (see \cite[Section 3.3]{bouchard2016general} or \cite[Definition 10]{della1981sur}). From \cite[Theorem 15]{della1981sur}, the family of random variables can be aggregated in a unique $\F$-optional process $V^i(a^{\star,-i},\xi^i,q)$ which satisfies
\[
V_t^i(a^{\star,-i},\xi^i,q)= \esssup_{a^i \in \Ac_i} \E^{\P^{a^i \otimes_i a^{\star,-i}}} \bigg[ \Uc_A \bigg( \xi^i - \int_t^T  \big( h_i(a_s^i) + c_i(q_s^i) \big) \mathrm{d}s \bigg)  \bigg | \Fc_t\bigg],\; \P\text{\rm--a.s.},\; t\in[0,T].
\]
Notice that the value at time $0$ of this process coincides with the value function of the agent. Define next, for every $a^i\in\Ac_i$, the following process which is a uniformly integrable $\P^{a^i \otimes_i a^{\star,-i}}$-supermartingale\footnote{Again, integrability follows from \eqref{eq:uniformly-integrable} and the fact that processes $a^i$ and $q$ being bounded.} and whose value at time $0$ is independent of $a^i$
\[
R_t^{a^i}(a^{\star,-i},\xi^i,q) := \mathrm{e}^{ \rho \int_0^t  ( h_i(a_s^i) + c_i(q_s^i)) \mathrm{d}s} V_t^i(a^{\star,-i},\xi^i,q),\; t\in[0,T].
\]
Then we have, for any $t\in[0,T]$
\begin{align*}
V_0^i(a^{\star,-i},\xi^i,q)  =R_0^{a^{\star,i}}(a^{\star,-i},\xi^i,q) & \geq \E^{\P^{a^{\star}}} [R_t^{a^{\star,i}}(a^{\star,-i},\xi^i,q)] \\
& \geq  \E^{\P^{a^{\star}}} [R_T^{a^{\star,i}}(a^{\star,-i},\xi^i,q)] \\
& = \E^{\P^{a^{\star}}} \bigg[ \Uc_A \bigg( \xi^i - \int_0^T \big( h_i(a_s^{\star,i}) + c_i(q_s^i) \big) \mathrm{d}s \bigg) \bigg]  = V_0^i(a^{\star,-i},\xi^i,q),
\end{align*}
and it follows that the process $R^{a^{\star,i}}(a^{\star,-i},\xi^i,q)$ is a uniformly integrable $\P^{a^{\star}}$-martingale. Since the filtration $\F$ satisfies the usual conditions, we consider from now on a c\`adl\`ag $\P$-modification of this process.

\medskip
From the multiplicative decomposition of negative martingales (see for instance \cite[Equation (8)]{yoeurp1976decomposition}) and the martingale representation theorem\footnote{Notice that $\P^{a^{\star}}$ satisfies the martingale representation property, due to Theorem III.5.24 in \cite{jacod2003limit}.}, there exists a predictable process $Z^i \in \mathbb{H}^2_\text{\rm loc}(\R)$ such that
\[
R_t^{a^{\star,i}}(a^{\star,-i},\xi^i,q) = V_0^i(a^{\star,-i},\xi^i,q) \Ec\bigg(  -\rho \int_0^t  Z_s^i \sigma \mathrm{d}W_s^{a^\star}\bigg).
\]
By applying It\^o's formula and using the definition of $R^{a^{\star,i}}(a^{\star,-i},\xi^i,q)$, we obtain that
\begin{align*}
\mathrm{d}V_t^i(a^{\star,-i},\xi^i,q) = &  \bigg( -\rho \big( h(a_t^{\star,i}) + c_i(q_t^i) \big) V_t^i(a^{\star,-i},\xi^i,q) + \rho V_t^i(a^{\star,-i},\xi^i,q)  Z_t^i \sum_{j=1}^N  (1-a^{\star,j}_t)p_j(q^j_t)   \bigg)\mathrm{d}t \\
& -\rho V_t^i(a^{\star,-i},\xi^i,q) Z_t^i\sigma \mathrm{d}W_t.
\end{align*}
Finally, defining $Y^{i} = -\frac{1}{\rho} \log(-V^i(a^{\star,-i},\xi^i,q))$, $Y:=(Y^1,\dots,Y^N)^\top$, $Z:=(Z^1,\dots,Z^N)^\top$ and applying It\^o's formula we have that the pair $(Y,Z)$ is a solution to BSDE \eqref{eq:BSDE}. Note also that $Z\in\Zc$ because $Y^{0,q,Z,i}=Y^i-Y_0^i$, with $Y_0^i\in\R$ and $\Uc_A(Y_t^i)=V_t^i(a^{\star,-i},\xi^i,q)$.

\medskip
To conclude, we show that $a^\star$ satisfies Condition \eqref{eq:optimal-effort}. Notice that from definition, for every $a^i\in\Ac_i$ we have
\begin{align*}
R_t^{a^i}(a^{\star,-i},\xi^i,q)  &= R_t^{a^{\star,i}}(a^{\star,-i},\xi^i,q) \mathrm{e}^{ -\rho \int_0^t  (h_i(a^{\star,i}_s)-h_i(a_s^i)) \mathrm{d}s}  \\
& = V_0^i(a^{\star,-i},\xi^i,q) - \int_0^t  \rho R_s^{a^{\star,i}} \mathrm{e}^{ -\rho \int_0^s  (h_i(a^{\star,i}_u)-h_i(a_u^i)) \mathrm{d}u}   Z_s^i  \big(\sigma \mathrm{d}W^{a^i\otimes_i a^{\star,-i}}_s- (a^i_s-a^{\star,i}_s)p_i(q^i_s) \mathrm{d}s\big) \\ 
&\quad  -  \int_0^t   \rho R_s^{a^{\star,i}} \mathrm{e}^{ -\rho \int_0^s  (h_i(a^{\star,i}_u)-h_i(a_u^i)) \mathrm{d}u}  (h_i(a^{\star,i}_s)-h_i(a_s^i)) \mathrm{d}s.
\end{align*}
Since $R^{a^i}(a^{\star,-i},\xi^i,q)$ is a (negative) $\P^{a^i \otimes_i a^{\star,-i}}$-supermartingale, we conclude that for every $a^i$ we have $\mathrm{d}t\otimes\mathrm{d}\P$\text{\rm--a.e.} on $[0,T]\times\Omega$
\[
h_i(a^{\star,i}_s) + Z_s^i a^{\star,i}_s p_i(q^i_s) \leq h_i(a_s^i) + Z_s^i a^i_s p_i(q^i_s),
\]
which implies \eqref{eq:optimal-effort}. 

\medskip
For the second part of the proof, let $(Y,Z)$ be a solution to BSDE \eqref{eq:BSDE} such that $Z\in\Zc$. Define the process $a^\star\in\Ac$, $\mathrm{d}t\otimes\mathrm{d}\P$\text{\rm--a.e.} on $[0,T]\times\Omega$, by
\[
a^{\star,i}_s := a^{\star,i}(q_s,Z_s), \; i\in\{1,\dots,N\}.
\]
Next, fix any firm $i$ and define for any effort $a^i\in\Ac_i$ the process
\[
U_t^{a^i} := \Uc_A\bigg(Y_t^i -\int_0^t \big( h_i(a_s^i) + c_i(q_s^i) \big) \textrm{d}s \bigg),\; t\in[0,T],
\]
which is a $\P^{a^i \otimes_i a^{\star,-i}}$-integrable process of class (D). Indeed, let $p$ be such as in \eqref{eq:classZ} and notice that $Y^i = Y^{0,q,Z,i}+Y_0^i$, with $Y_0^i\in\R$. Then we have, for every $\hat p_1\in(1,p)$
\begin{align*}
\E^{\P^i}\Big[ \big|U_\tau^{a^i} |^{p_1} \Big]  & \leq \E\Big[ \big|U_\tau^{a^i} |^p \Big]^\frac{p_1}{p}  \E\bigg[ \bigg( \frac{\textrm{d}\P^i}{\textrm{d}\P} \bigg)^{p_2} \bigg]^\frac{1}{p_2}  
 \leq \E\Big[ \big|\Uc_A(Y_\tau^{0,q,Z,i})\big|^p \Big]^\frac{p_1}{p}       \mathrm{e}^{\rho T(\bar{h} + \bar{c})p_1 }      \mathrm{e}^{-\rho Y_0^i p_1}     \mathrm{e}^{\frac{(p_2-1)TN^2\bar{p}^2 }{2 \sigma^2}}  ,
\end{align*}
where we denoted $\P^i = \P^{a^i \otimes_i a^{\star,-i}}$, and $p_2$ is the H\"older's conjugate of $\frac{p}{p_1}$. It follows then that
\[
\sup_{\tau \in\Tc_{0,T}} \E^{\P^{a^i \otimes_i a^{\star,-i}}}\Big[ \big| U_\tau^{a^i} \big|^{p_1} \Big] < \infty.
\] 
By It\^o's formula, we have
\begin{align*}
\mathrm{d}U_t^{a^i} 
& = -\rho U_t^{a^i} \big( h_i(a^{\star,i}_t) + Z_t^i a^{\star,i}_t p_i(q^i_t) - h_i(a_t^i) - Z_t^i  a^i_t p_i(q^i_t) \big) \textrm{d}t - \rho U_t^{a^i} \sigma Z_t^i \textrm{d}W_t^{a^i\otimes_i a^{\star,-i}}.
\end{align*}

It follows then that that the process $U^{a^i}$ is a $\P^{a^i \otimes_i a^{\star,-i}}$-local supermartingale\footnote{It is the stochastic exponential of a continuous semimartingale.} of class (D), hence a supermartingale,
whose value at time $0$ is independent of $a^i$. By the same argument, we have that 
$U^{a^{\star,i}}$ is a $\P^{a^{\star}}$-martingale
, so that
\[
U_0^i(a^i,a^{\star,-i},\xi^i,q) = \mathbb{E}^{\P^{a^i \otimes_i a^{\star,-i}}}\big[ U_T^{a^i} \big] \leq U_0^{a^i} = U_0^{a^{\star,i}} = \mathbb{E}^{\P^{a^{\star}}}\big[ U_T^{a^{\star,i}} \big] = U_0^i(a^{\star,i},a^{\star,-i},\xi^i,q),
\]
which means the action $a^{\star,i}$ is the best-reaction for producer $i$, given the action of the others $a^{\star,-i}$.

\end{proof}

\begin{proof}[Proof of {\rm Proposition \ref{prop:contracts}}]

It follows from the proof of \Cref{prop:optimal-effort}$(ii)$, since the pair $(Y^{y,q,Z},Z)$ is a solution to BSDE \eqref{eq:BSDE}, with $Z\in\Zc$. It is also proved that for every $a^i\in\Ac_i$ $ U_0^i(a^i,a^{\star,-i},\xi^i,q) \leq U_0^i(a^{\star,i},a^{\star,-i},\xi^i,q)=\Uc_A(y^i) $. This implies that $a^\star$ is a Nash equilibrium and $V_0^i(a^{\star,-i},\xi^i,q)=\Uc_A(y^i) $. Moreover, equality is attained only at the control $a^{\star,i}$ because of the uniqueness of the minimiser in \eqref{eq:a-star-appendix}.

\end{proof}

\section{Value function of the ISO} \label{app:iso}

\begin{proof}[Proof of {\rm Proposition \ref{prop:value-function-2}}]

Note that, for any $(q,\phi, y,Z)\in\hat\Cc$, the dependence of the objective function on the process $Y^{y,q,Z}$ is only through the sum of the initial values, that is the coordinates of $y\in\mathbb{R}^n$. By plugging $Y_T^{y,q,Z}$ into \eqref{eq:V0-first}, it readily follows then 
\begin{align*}
V_0 
&= \inf_{(q,\phi,Z)\in\Pc\times \Zc}   \E^{\P^{a^\star(q,Z)}} \bigg[ \sum_{i=1}^N \int_0^T  \left( h_i(a^{\star,i}(q_s^i,Z_s^i))+\frac{\rho\sigma^2}{2} (Z_s^i)^2+ 2c_i(q^i_s)  \right) \mathrm{d}s + \int_0^T \Lambda( L_s - \ell_0) \mathrm{d}s  \bigg] - \sum_{i=1}^N \frac{1}{\rho}\log(-R_0^i).
\end{align*}

\end{proof}

In the following, we list a series of propositions which will result in Theorem \ref{thr-iso-main}.

\begin{proposition}\label{prop-solution-pde}
There exists a unique viscosity solution to the {\rm HJB} equation \eqref{eq:pde-value-function} which has at most polynomial growth at infinity. Such solution is continuously differentiable in the space variable, with bounded derivative.
\end{proposition}

\begin{proof}[Proof of {\rm Proposition \ref{prop-solution-pde}}]  

$(i)$ We start with the uniqueness of such solutions. Define $\tilde F:\R\times\R \longrightarrow \R$ by 
\[
\tilde F(\ell,\alpha)  = \inf_{(z,q,\phi)\in\R^{N}\times \hat{P} } \bigg\{  \sum_{i=1}^N \bigg( \alpha  (1-a^{\star,i}(q,z)) \frac{p_i(q^i)}{\sigma} +  h_i(a^{\star,i}(q,z))+\frac{\rho\sigma^2}{2} (z^i)^2+2c_i(q^i) \bigg)\bigg\}  + \Lambda(\ell - \ell_0)  .
\]
We have $|\tilde F(\ell,0)| \leq N (\bar{h} + \bar{c}) + |\Lambda(\ell - \ell_0)|$, with $\bar{h}$ and $\bar{c}$ defined in \eqref{eq:constants}. Thus $\tilde F(\ell,0)$ has polynomial growth, and
\begin{align*}
|\tilde F(\ell,\alpha_1) - \tilde F(\ell,\alpha_2)| & = \bigg| \inf_{(z,q,\phi)\in\R^{N}\times \hat{P} } g(\alpha_1 / \sigma,z,q,\phi) - \inf_{(z,q,\phi)\in\R^{N}\times \hat{P} } g(\alpha_2 / \sigma,z,q,\phi)   \bigg| \\
& \leq \sup_{(z,q,\phi)\in\R^{N}\times \hat{P} } \big |  g(\alpha_1/ \sigma,z,q,\phi) -  g(\alpha_2/ \sigma,z,q,\phi) \big|  \leq |\alpha_1 - \alpha_2 | N \frac{\bar{p}}{\sigma},
\end{align*}
where $\bar{p}$ is defined in \eqref{eq:constants}. Therefore, Assumption (A2) from \citeauthor*{pardoux2014stochastic} \cite{pardoux2014stochastic} is satisfied and the uniqueness of viscosity solutions to \eqref{eq:pde-value-function} with polynomial growth follows from \cite[Theorem 6.106]{pardoux2014stochastic}. 

\medskip
$(ii)$ For the existence, we start by proving that $\tilde F$ is continuously differentiable in $\alpha$, with bounded derivative. Let us denote by $\tilde \mu: \R \rightrightarrows \R^N \times \hat P$ the correspondence of optimisers in $g$, as follows
\[
\tilde \mu (\alpha) = \bigg\{ (\tilde z,\tilde q,\tilde \phi)\in\R^{N}\times \hat{P} :  g(\alpha/\sigma,\tilde z,\tilde q,\tilde \phi) = \inf_{(z,q,\phi)\in\R^{N}\times \hat{P} } \big\{g(\alpha/\sigma,z,q,\phi)\big\}   \bigg\}.
\]
Notice that the optimisation over each $z^i$ can be reduced to a compact set since the objective function is the sum of a bounded and a quadratic function, namely we have $z^i \in \big[-(2|\alpha| \overline{p}/\sigma+\overline{h}+2\overline{c})N , (2 |\alpha|\overline{p}/\sigma+\overline{h}+2\overline{c})N\big],$
where $\overline{p}$, $\overline{h}$ and $\overline{c}$ are defined in \eqref{eq:constants}. Since the correspondence $\alpha \rightrightarrows  \big[-(2|\alpha| \overline{p}/\sigma+\overline{h}+2\overline{c})N, (2|\alpha|\overline{p}/\sigma+\overline{h}+2\overline{c})N\big]$ is continuous and compact-valued, it follows from the maximum Theorem (see \cite[Theorem 17.31]{aliprantis2006infinite}) that $\tilde \mu$ has nonempty compact values and it is upper-hemicontinuous. By \citeauthor*{milgrom2002envelope} \cite[Theorem 2]{milgrom2002envelope} we have that $\tilde F$ is absolutely continuous in $\alpha$ and we have, for any selection $(\tilde z(\alpha),\tilde q(\alpha),\tilde \phi(\alpha)) \in \tilde \mu(\alpha)$ 
\[
\tilde F'(\ell,\alpha) = \sum_{i=1}^N \big(1-a^{\star,i}(\tilde q(\alpha),\tilde z(\alpha))\big)\frac{p_i(\tilde q^{i}(\alpha))}{\sigma}.
\]
Since the maps $a^{\star,i}$ and $p_i$ are uniformly continuous over $A_i$ and $[0,Q^i]$ respectively, and $\tilde \mu $ is upper-hemicontinuous, it follows that $\tilde F'(\ell,\cdot)$ is continuous and bounded.

\medskip
$(iii)$ By the previous point, and the fact that the terminal condition in \eqref{eq:pde-value-function} is null, we have that the assumptions in \citeauthor*{ma2002representation} \cite[Theorem 3.1]{ma2002representation} as well as their Assumption (A1) are satisfied. It follows that the function $\tilde v(t,\ell) := \tilde Y_t^{t,\ell}$, where $(\tilde L^{t,\ell},\tilde Y^{t,\ell}, \tilde Z^{t,\ell})$ is the adapted solution to the FBSDE
\begin{align*}
\tilde L_s^{t,\ell} = \ell + \int_t^s \sigma \mathrm{d}W_r,  \;
\tilde Y_s^{t,\ell} &= \int_s^T \tilde F(\tilde L_r^{t,\ell},Z_r^{t,\ell} ) \mathrm{d}r - \int_s^T \tilde Z_r^{t,\ell} \mathrm{d}W_r, 
\end{align*}
is a viscosity solution to the PDE. Moreover, $\partial_\ell \tilde v$ exists for every $(t,\ell)\in [0,T]\times \mathbb{R}$ and is continuous. By \cite[Corollary 3.2]{ma2002representation}, $\tilde v(t,\cdot)$ has polynomial growth at infinity and $\partial_\ell \tilde v$ is bounded.
\end{proof}

To make a link between the viscosity solution given by the previous proposition and the value of problem \eqref{eq:V0-one-dimensional}, we define the value function of the ISO starting from any $(t,\ell)\in[0,T]\times\R$
\begin{equation}\label{eq:value-function-dynamic}
V(t,\ell) :=  \inf_{(q,\phi,Z)\in\Pc\times \Zc}   \E^{\P^{a^\star(q,Z)}} \bigg[  \int_t^T  \sum_{i=1}^N \left( h_i(a^{\star,i}(q_s^i,Z_s^i))+\frac{\rho\sigma^2}{2} (Z_s^i)^2+ 2c_i(q^i_s)  \right) \mathrm{d}s + \int_t^T \Lambda( L_s^{t,\ell} - \ell_0) \mathrm{d}s  \bigg], 
\end{equation}
where $L^{t,\ell}$ is the solution to 
\[
 L_s^{t,\ell} = \ell  + \int_t^s \sum_{i=1}^N (1-a^{\star,i}(q_u^i,Z_u^i))p_i(q^i_u)  \mathrm{d}u + \sigma \int_t^s \textrm{d}W^{a^\star(q,Z)}_u, \; s\in[ t,T].
\]
Notice that by definition we have the equality $\hat V_0 = V(0,L_0)$, for the value of problem \eqref{eq:V0-one-dimensional}.

\begin{proposition}\label{prop:polynomial-growth}
The function $V$ defined in \eqref{eq:value-function-dynamic} is a viscosity solution of the {\rm HJB} equation \eqref{eq:pde-value-function} and $V(t,\cdot)$ has polynomial growth at infinity. Consequently, $V_\ell$ exists and it is bounded.
\end{proposition}

\begin{proof}[Proof of {\rm Proposition \ref{prop:polynomial-growth}}]

Since $V(t,\ell)$ is an infimum, it has an upper bound given by choosing the controls $Z \equiv 0$ and $(q,\phi) \equiv (q_0,\phi_0)$, where $(q_0,\phi_0)$ is an arbitrary element of $\hat{P}$. We have therefore
\[
V(t,\ell) \leq   \E^{\P^{a^\star(q_0,0)}}  \bigg[ \sum_{i=1}^N  \big( h_i(a^{\star,i}(q_0^i,0))+ 2c_i(q^i_0)  \big)(T-t) + \int_t^T \Lambda( L_s^{t,\ell} - \ell_0) \mathrm{d}s  \bigg].
\]
From \cite[Theorem 1.3.15]{pham2009continuous}, we have for every $(q,Z)\in\Pc\times \Zc$, $\E^{\P^{a^\star(q,Z)}}  \big[\sup_{t \leq s \leq T} |L_s^{t,\ell}|^2 \big] \leq K_0(1+ \ell^2),$
with a constant $K_0>0$. This implies that for some $K_1>0$
\[
\E^{\P^{a^\star(q_0,0)}} \bigg[  \int_t^T | \Lambda( L_s^{t,\ell} - \ell_0)|  \mathrm{d}s  \bigg] \leq  \E^{\P^{a^\star(q_0,0)}} \bigg[  \int_t^T K_\Lambda(1+| L_s^{t,\ell} - \ell_0|^2)  \mathrm{d}s  \bigg] \leq K_1(1+(\ell-\ell_0)^2).
\]
It is clear that $V$ is bounded by below so we conclude that $V$ has quadratic growth. The fact that $V$ is a viscosity solution to the PDE \eqref{eq:pde-value-function} is standard in stochastic control. For instance, we can use \citeauthor*{pham2009continuous} \cite[Propositions 4.3.1 and 4.3.2]{pham2009continuous}, since the function $G$ is finite-valued. To conclude, from the uniqueness of viscosity solutions to equation \eqref{eq:pde-value-function} we obtain that $V$ must coincide with the function in Proposition \ref{prop-solution-pde}. Hence, $V_\ell$ exists and it is bounded.
\end{proof}

We turn now our attention to the optimal controls in problem \eqref{eq:V0-one-dimensional}. Let us denote by $\mu^\star: \R \rightrightarrows \R^N \times \hat P$ the correspondence of optimisers in $g$, for given $\alpha$, that is 
\[
\mu^\star (\alpha) = \bigg\{ (z^\star, q^\star, \phi^\star)\in\R^{N}\times \hat{P} :  g(\alpha,z^\star,q^\star,\phi^\star) = \inf_{(z,q,\phi)\in\R^{N}\times \hat{P} } g(\alpha,z,q,\phi)   \bigg\}.
\]

\begin{lemma} \label{lemma-optimizers}
The correspondence $\mu^\star$ has nonempty compact values and it is upper-hemicontinuous.
\end{lemma}

\begin{proof}[Proof of {\rm Lemma \ref{lemma-optimizers}}]
The proof is identical to part $(ii)$ in the proof of \Cref{prop-solution-pde}. The optimisation over each $z^i$ can be reduced to a compact set, namely we have $z^i \in \big[-(2|\alpha| \overline{p}+\overline{h}+2\overline{c})N , (2 |\alpha|\overline{p}+\overline{h}+2\overline{c})N\big].$
From the maximum Theorem \cite[Theorem 17.31]{aliprantis2006infinite}, $ \mu^\star$ has nonempty compact values and it is upper-hemicontinuous.
\end{proof}

We can finally prove \Cref{thr-iso-main} as a direct consequence of the results in this section.

\begin{proof}[Proof of {\rm Theorem \ref{thr-iso-main}}]

$(i)$ We have that $\hat V_0 = V(0,L_0)$. Result is immediate from Propositions \ref{prop-solution-pde} and \ref{prop:polynomial-growth}.

\medskip
$(ii)$ First, the control $(q^\star,\phi^\star,Z^\star)\in\Pc\times\Zc$. Indeed, from the proof of \Cref{lemma-optimizers} we have that $Z^\star$ is bounded which implies that $Z^\star\in\Zc$. Second, since the controls $(q^\star,\phi^\star,Z^\star)$ attain the minimum in the Hamiltonian function $G$, they are the optimal controls in the reformulated problem of the ISO.

\medskip
$(iii)$ It follows from \Cref{th:main}, that the optimal remuneration is given by $Y_T^{y,Z^\star}$.
\end{proof}

\section{Simpler problem for the ISO} \label{ap:example}

\begin{proof}[Proof of Proposition \ref{prop-continuity-sec3}]
We have deduced that $v^d(q,\phi)=\hat v^d(q,\phi)  - \sum_{i=1}^N \frac{1}{r}\log(-R_0^i)$ and  $\hat v^d(q,\phi)= Y^{q,\phi,0,L_0}_0$, where $(L^{q,\phi,0,L_0},Y^{q,\phi,0,L_0}, Z^{q,\phi,0,L_0})$ is the adapted solution to the FBSDE
\begin{align*}
 L_s^{0,L_0} = L_0 + \int_0^s \sigma \mathrm{d}W_r,  \;
 Y_s^{q,\phi,0,L_0} &= \int_s^T \tilde F^{q,\phi}( L_r^{0,L_0},Z_r^{q,\phi,0,L_0} ) \mathrm{d}r - \int_s^T Z_r^{q,\phi,0,L_0} \mathrm{d}W_r, 
\end{align*}
with $\tilde{F}^{q,\phi}:\R\times\R \longrightarrow\R$ defined by
\[
\tilde F^{q,\phi}(\ell,\alpha)  = \inf_{z \in\R^{N} } \bigg\{  \sum_{i=1}^N \bigg( \alpha  (1-a^{\star,i}(q,z)) \frac{p_i(q^i)}{\sigma} +  h_i(a^{\star,i}(q,z))+\frac{\rho\sigma^2}{2} (z^i)^2+2c_i(q^i) \bigg)\bigg\}  + \Lambda(\ell - \ell_0).
\]
As in the proof of Proposition \ref{prop-solution-pde}, we check that
\[
| \tilde F^{q,\phi}(\ell,\alpha_1) - \tilde F^{q,\phi}(\ell,\alpha_2) |   \leq   |\alpha_1 - \alpha_2| N \frac{\bar{p}}{\sigma}, \quad \forall (\ell,q,\phi)\in \R\times\hat P, \forall \alpha_1,\alpha_2 \in \R,
\]
and for every $(\ell,\alpha)\in\R^2, \forall (q_1,\phi_1), (q_2,\phi_2)\in\hat P$
\[
 | \tilde F^{q_1,\phi_1}(\ell,\alpha) - \tilde F^{q_2,\phi_2}(\ell,\alpha) |     \leq   \sup_{z\in\R^N} \sum_{i=1}^N  \bigg| \alpha  (1-a^{1,i}(z)) \frac{p_i(q_1^i)}{\sigma}  - \alpha  (1-a^{2,i}(z)) \frac{p_i(q_2^i)}{\sigma}  +  \Delta h_i(z) +2 \Delta c_i  \bigg|, 
\]
with $a^{1,i}(z)= a^{\star,i}(q_1,z)$, $a^{2,i}(z)= a^{\star,i}(q_2,z)$, $\Delta h_i(z) = h_i(a^{1,i}(z)) - h_i(a^{2,i}(z))$, $\Delta c_i = c_i(q_1^i) - c_i(q_2^i)$. By restricting the supremum over each $z^i$ to a compact set, as in the proof of Lemma \ref{lemma-optimizers}, and by the uniform continuity of all the maps, it follows that  $ | \tilde F^{q_1,\phi_1}(\ell,\alpha) - \tilde F^{q_2,\phi_2}(\ell,\alpha) |  \to 0$ as $(q_1,\phi_1) \to (q_2,\phi_2)$. Then, the hypothesis in \cite[Proposition 2.4]{el1997backward} are satisfied so we conclude that $\hat v^d$ is continuous. The existence of a minimizer $(q^\star,\phi^\star)\in\hat P$ for problem \eqref{eq:ISO-sec3} holds because $\hat P$ is compact.

\end{proof}

\begin{proof}[Proof of Proposition \ref{prop:explicit-example}]

Under the given assumptions, the optimal efforts for the providers are given by
\[
a^{\star,i}_s\in\argmin_{a\in A_i} \bigg\{\frac{h_i}{2}a^2 -Z_s^i p_i (1-a)q^i_s \bigg\} \Longrightarrow a^{\star,i}(z)=  \Pi_{A_i}\bigg( - \frac{z^i p_i q_i}{h_i } \bigg).
\]
Optimising over $z$, we see that the values for which $- \frac{z^i p_i q_i}{h_i }$ is outside of $A_i$ are not optimal so we can compute the infimum in $\hat G$ directly by replacing $a^{\star,i}(z)=  - \frac{z^i p_i q_i}{h_i }$.
 Let $A_i=[0,A_i]$ and $\hat A_i :=-\frac{h_i A_i}{p_i q_i} $, then
\begin{align*}
\hat G & =  \sum_{i=1}^N \bigg( \inf_{z^i\in[\hat A_i,0 ]}  \bigg\{  \frac{ \alpha  z^i (p_i)^2 (q_i)^2}{h_i }  + \frac{( z^i)^2 (p_i)^2 (q_i)^2}{2h_i } +\frac{\rho\sigma^2}{2} (z^i)^2   \bigg\}  + \alpha  p_i q_i + 2c_i(q_i) \bigg) +\frac12\gamma\sigma^2+ \lambda (\ell-\ell_o)^n .
\end{align*}
We have the optimal  $z^i=\Pi_{[\hat A_i,0]}\left(\frac{-\alpha (p_i q_i)^2}{\sigma^2 \rho h_i + (p_i q_i)^2}\right)$ and, recalling $M_i=\frac{A_i h_i(\sigma^2 \rho h_i + (p_iq_i)^2)}{(p_iq_i)^3}$, we get
\begin{align}\label{eq:G-explicit-example}
\hat G  & = \sum_{i=1}^N  \bigg(   \alpha   p_i q_i   +  2c_i(q_i) - \frac{ \alpha^2 (p_i q_i)^4}{2 h_i(\sigma^2 \rho h_i + (p_i q_i)^2) } \mathbf{1}_{\{\alpha\in[0,M_i]\}}   + \hat A_i \bigg(\frac{2\alpha (p_i q_i)^2 + \hat A_i(\sigma^2 \rho h_i + (p_i q_i)^2)}{2h_i} \bigg) \mathbf{1}_{\{\alpha \geq M_i \} }\bigg) \\
& \nonumber ~~~ + \frac12\gamma\sigma^2+ \lambda (\ell-\ell_o).
\end{align}
Assume for a while that the solution is such that $v_\ell \in [0,M_i]$, so then the PDE becomes 
\[
v_t + \mu  v_\ell + B   v_{\ell\ell} - C(v_\ell)^2 + D +  \lambda (\ell-\ell_o)^n = 0, ~ v(T,\ell)=0,
\]
with 
\[
\mu :=\sum_{i=1}^N p_i q_i,~ B:=\frac{1}{2}\sigma^2, ~C:=\sum_{i=1}^N \frac{(p_i q_i)^4}{2h_i(\sigma^2 \rho h_i + (p_i q_i)^2)}, ~D:=2 \sum_{i=1}^N c_i(q_i).
\]
Let $u=e^{-\frac{C}{B} v}$, then $u$ satisfies the PDE (provided that $u_\ell/u$ is bounded)
\[
u_t + \mu u_\ell + B  u_{\ell\ell} - \frac{CD}{B} u - \frac{\lambda C}{B}   u (\ell-\ell_o)^n = 0, ~ u(T,\ell)=1. 
\]
Therefore we have from the Feynman–Kac formula $u(t,\ell) = \E^\Q\big[  \textrm{e}^{\frac{-C}{B}\int_t^T(D+\lambda(L_s-\ell_0)^n) \textrm{d}s}   \big| L_t = \ell \big],$
under the measure $\Q$ such that $\textrm{d} L_t =   \sum_{i=1}^N p_i q_i \textrm{d}t + \sigma \textrm{d}W^\Q_t.$ Then we have
\[
u(t,\ell) = \textrm{e}^{\frac{-CD}{B}(T-t)} \E^\Q\bigg[   \textrm{e}^{\frac{-C\lambda}{B} \int_0^{T-t} \big( \ell - \ell_0 + \mu s + \sigma W^\Q_s\big)^n    \textrm{d}s }  \bigg] = \textrm{e}^{\frac{-CD}{B}(T-t)} \E^\Q\bigg[   \textrm{e}^{-\hat\gamma \int_0^{T-t} W_s^{\Q,\hat\mu}   \textrm{d}s }  \bigg],
\]
with $\hat\gamma: =\frac{C\lambda\sigma}{B}$, $\hat\mu := \frac{\mu}{\sigma}$ and $W_s^{\Q,\hat\mu}:=\hat\mu s + W_s^\Q$ with $W_0^\Q=\frac{\ell-\ell_0}{\sigma}$. Then we have from the result 2.1.8.3 (page 262) in \citeauthor*{borodin2002handbook} \cite{borodin2002handbook}
\[
u(t,\ell) = \textrm{e}^{\frac{-CD}{B}(T-t)} \textrm{e}^{-\frac{C\lambda}{B}(\ell-\ell_0)(T-t)} \textrm{e}^{-\frac{C\lambda\mu}{2B}(T-t)^2} \textrm{e}^{\frac{C^2\lambda^2}{3B}(T-t)^3}.
\]
Notice that $\frac{u_\ell}{u} = -\frac{C\lambda}{B}(T-t)$, so then $v_\ell = \lambda(T-t)\in [0,M_i]$ due to our assumptions. Indeed, we can also compute the value function of the ISO
\[
v(t,\ell)= -\frac{B}{C}\ln(u) = D(T-t) + \lambda(\ell-\ell_0)(T-t) + \frac{\lambda\mu}{2}(T-t)^2-\frac{C\lambda^2}{3}(T-t)^3.
\]
Since we found a solution to PDE \eqref{eq:pde-value-function-sec3} with polynomial growth, we know from Theorem \ref{thr-iso-main-sec3} that we have the value of the problem with fixed plan by $\hat v^d(q,\phi)=v(0,L_0)$. It follows that the optimal process $Z$ and the optimal efforts are
\[
Z_t^{\star,i}=\frac{- \lambda(T-t) (p_i q_i)^2}{\sigma^2 \rho h_i + (p_i q_i)^2} = (t-T)z_i,\; a^{\star,i}_t =  - \frac{Z^{i,\star}_t p_i q_i}{h_i } = \frac{  \lambda(T-t) (p_i q_i)^3}{h_i(\sigma^2 \rho h_i + (p_i q_i)^2)} = \frac{(T-t) z_i p_i q_i}{h_i }.
\]
Finally, we see that the optimal contract is given by
\[
\xi^{\star,i} = y^i  + \int_0^T f^i(s,Z^{\star,i}_s)  \mathrm{d}s + \int_0^T Z^{\star,i}_s \mathrm{d}L_s,
\] 
where
\[
f^i(s,Z^{\star,i}_s) = \frac{ h_i}{2}(a_s^{\star,i})^2 + \frac{\rho\sigma^2}{2} (Z^{\star,i}_s)^2  + c_i(q_i)  - Z^{\star,i}_s \sum_{i=1}^N p_i(1-a_s^{\star,i})q_i .
\]
\end{proof}

\end{appendix}

\end{document}